\documentclass[11pt]{article}
\usepackage{amsmath, amsfonts, amsthm, amssymb,verbatim,color,a4wide}

        \newtheorem{theorem}{Theorem}[section]
\newtheorem{definition}[theorem]{Definition}
        \newtheorem{lemma}[theorem]{Lemma}
        
        \newtheorem{remark}[theorem]{Remark} 
\numberwithin{equation}{section}
\newcommand \bel {\begin{equation}\label}

\newcommand \ee {\end{equation}}

\newcommand \alphab {\overline\alpha}  
\newcommand \ub {\overline u}
\newcommand \vb {\overline v} 
\newcommand \ubar   	{{\overline u}}
\newcommand \vbar   	{{\overline v}}
\newcommand \wbar   	{{\overline w}}

\newcommand \Kbar {\overline K} 
\newcommand \del		\partial
 
\newcommand{\auth}{\textsc}
\DeclareMathOperator	\sgn  {sgn} 
 \newcommand \la 		\langle
\newcommand \ra 		\rangle
\newcommand \be   	{\begin{equation}} 
 
\newcommand \RR      {\mathbb{R}}
\newcommand \eps     \epsilon

\newcommand \Omegabf {\mbox{\boldmath $\Omega$}}
\newcommand \Tau {\mathcal T} 

\def\XXint#1#2#3{{\setbox0=\hbox{$#1{#2#3}{\int}$}
\vcenter{\hbox{$#2#3$}}\kern-.5\wd0}}

\newcommand \CL 		{C^\infty \Lambda} 
\newcommand \DL 		{{\mathcal D} \Lambda} 
\newcommand \Dcal 	{\mathcal D} 
\newcommand \Lloc 	{L^1_{\text{loc}}}
\newcommand \Mring 	{\mathring M} 
\newcommand \Hcal   {\mathcal{H}}
\newcommand \Bcal   {\mathcal{B}}
\newcommand \len    {\lambda_{K,e^0}}
\newcommand \tu     {{\tilde u_{K,e^0}}}
\newcommand \uep    {{u_{e_K^+}}}
\newcommand \Qbf    {Q^{\Omegabf} }
\newcommand \bu     {{\bar u_{K,e^0}}}

\begin{document}

\setcounter{footnote}{-1}

\title{Formulation and convergence of the finite volume method
\\
 for conservation laws on spacetimes with boundary}

\author{
%                Correct spelling:   LeFLOCH  or  LeFloch
Jan Giesselmann$^1$ and Philippe G. LeFloch$^2$ 
\footnote{
$^1$ Fachbereich Mathematik, TU Darmstadt, Dolivostr. 15, D-64293 Darmstadt, Germany.
\newline 
E-mail: {\tt giesselmann@mathematik.tu-darmstadt.de} 
\newline						
						Laboratoire Jacques-Louis Lions \& Centre National de la Recherche Scientifique,
Sorbonne Universit\'e, 4 Place Jussieu,  75252 Paris, France.
\newline 
E-mail : {\tt contact@philippelefloch.org}
\newline
\textit{\ AMS Subject Class.} {Primary: 35L65. Secondary: 76L05, 76N.} 
\, 
\textit{Key words and phrases.} Hyperbolic conservation law; flux field of differential forms;  spacetime with boundary; global hyperbolicity; entropy solution; finite volume. 
\hfill To appear in: Numerische Mathematik (2020).
}}

\date{Revised version: May 2019}
\maketitle
\begin{abstract}
We study nonlinear hyperbolic conservation laws posed on a differential $(n+1)$-manifold with boundary referred to as a spacetime,
and defined from a prescribed flux field of $n$-forms depending on a parameter
(the unknown variable) ---a class of equations proposed by LeFloch and Okutmustur in 2008. Our main result is a proof of the convergence of the finite volume method for weak solutions satisfying suitable entropy inequalities. 
A main difference with previous work is that we allow for slices with a boundary and, in addition, introduce a new formulation of
the finite volume method involving the notion of total flux functions.   
Under a natural global hyperbolicity condition on the flux field and the spacetime and by assuming that the spacetime admits a foliation by compact slices with boundary, we establish an existence and uniqueness theory for the initial and boundary value problem, and we prove a contraction property in a geometrically natural $L^1$-type distance. 
\end{abstract}

%==================================================================================================================

\section{Introduction} 

The mathematical study of weak solutions to hyperbolic conservation laws on curved manifolds was initiated by LeFloch and co-authors about ten years ago \cite{ABL,ALO,BL07,LeFlochOkutmustur2}, and this subject extended in several directions \cite{BL,BLM1,BCHL09,BFL09,Gie09,GM14,KMS15,LM14,LM13}. The main motivation comes from geophysical fluid dynamics (in which the problems are posed on a curved surface such as the sphere) and general relativity (in which the Einstein-Euler equations  are expressed on a manifold whose metric is one of the unknowns).  
The class of hyperbolic conservation laws provides a simplified, yet very challenging, model on which one can develop fundamental techniques and understand aspects of nonlinear wave propagation and shock formation on manifolds.

In the present paper, we built upon the work by LeFloch and Okutmustur \cite{LeFlochOkutmustur2}  and consider conservation laws whose flux is expressed as a family of differential forms. 
Let $M$ be an oriented, smooth $(n+1)$-dimensional manifold endowed with a smooth family of $n$-forms $\omega=\omega(u)$ which we refered to as a {\it flux field}.
The conservation law of interest then reads 
\bel{claw}
 d(\omega(u))=0,
\ee
where $d$ denotes the exterior derivative and in which the unknown $u: M \to \RR$ is a scalar field defined on this manifold.

The study of weak solutions to hyperbolic equations on manifolds was initiated in Ben-Artzi and LeFloch~\cite{BL07} for the class of equations
\bel{Riemann}
 \del_t u + \operatorname{div}_g (f(x,u)) =0, \qquad M= \mathbb{R}_+ \times N, 
\ee
where $N$ is a closed, oriented manifold endowed with a Riemannian metric $g$ and 
$f$ is a prescribed family of flux vector fields on $N$.
In a local coordinate system $(t,x^1,\dots, x^n)$ associated with the basis of tangent vectors $\del_t, \del_{x^1},\dots, \del_{x^n}$, we can express any vector field $f$ by its coefficients $f_i$, i.e.~one sets $f(u)=\sum_i f_i(u) \del_{x^i}$.
The Riemannian metric can be expressed as $g_{ij}:= g(\del_{x^i},\del_{x^j})$, so that 
\[ 
 \operatorname{div}_g f := \frac{1}{\sqrt{|g|}} \del_{x^i} \big( f_i \sqrt{|g|}\big),
\]
where $|g|$ denotes the absolute value of the determinant of $(g_{ij})$.

Observe that we recover the 
following generalization of \eqref{Riemann}
\bel{Riemann2}
 \del_t h(u) + \operatorname{div}_g (f(x,u)) =0, \qquad M= \mathbb{R}_+ \times N, 
\ee
from \eqref{claw} by choosing
\bel{fluxformR} \omega(u)= h(u)  \sqrt{|g|}  dx^1 \wedge \dots \wedge dx^n + \sum_{i=1}^n (-1)^i f_i(u) \sqrt{|g|} dt \wedge dx^1 \wedge \cdots \wedge \widehat{dx^i} \wedge  \dots \wedge dx^n. 
\ee
Hence, the field $\omega$ combines information on the flux field $f$, the conserved variables $h$ and the volume form  $\sqrt{|g|}  dx^1 \wedge \cdots  \wedge dx^n$ induced by $g$. 
In particular, no separate information on the geometric structure of $M$ is required and the knowledge of the flux field is sufficient for writing \eqref{claw}. 
This also enables us to define a suitable notion of hyperbolicty in a unifying way for a large class of problems.
As pointed out in \cite{LeFlochOkutmustur2}, using the general form \eqref{claw} instead of the corresponding formulations for Lorentzian or Riemannian manifolds is not only more general but also
allows one to develop a mathematical theory which is conceptually and technically simpler.
From a conceptual point of view it is desirable to have a unifying theory (as the one presented here) encompassing  conservation laws on (time dependent) Riemannian manifolds, 
conservation laws on Lorentzian manifolds and hyperbolic problems in which time and space derivatives act on non-linear functions of the state variable $u$.

Our main objective is to further extend the results in \cite{LeFlochOkutmustur2} and allow for time slices {\sl with a non-empty boundary}, i.e.,
there is a part of the boundary of $M$ which does not constitute 'initial data' for the Cauchy problem and on which 'boundary conditions' must be prescribed.
Recall that imposing boundary conditions for nonlinear hyperbolic conservation laws is a delicate matter, since the nature of the data to be prescribed depends on the (unknown) solution.
In fact, the boundary condition needs to be included in the very definition of entropy solutions, in a suitably weak sense. This issue was first discussed in \cite{BLN} in the class of functions of bounded total variation and later in
\cite{Szepessy} for measure-valued solutions.
We follow an approach initiated in the Euclidean setting by LeFloch and Dubois \cite{DL88} and later  \cite{CCL94,CCL95,KL}. An alternative approach to boundary conditions for hyperbolic conservation laws goes back to Otto \cite{Otto}.
Recall that the theory of initial value problems for entropy solutions to scalar conservation laws in the Euclidean space goes back to Kruzkov 
\cite{Kruzkov} and DiPerna \cite{DiPerna} and the present work relies on their pioneering contributions.
 We also note that there is only one previous paper on the analyis of hyperbolic conservation laws on manifolds with boundary, \cite{KMS15}, which studied problems posed on stationary Riemannian manifolds.
 Let us mention conservation laws on domains of outer communication of  black holes, \cite{LX16}, as a motivation for the treatment of conservation laws on spacetimes with boundary.

We will introduce here suitable notions of weak solutions and measure-valued solutions (in the sense of DiPerna).  
Our definitions make sense for general flux fields $\omega$ but, in addition, an entropy condition is imposed which singles out the relevant weak solutions. To this end, we impose that the pair $(M,\omega)$ satisfies a ``global hyperbolicity condition'', as we call it, which in particular 
provides a global time orientation and allows us to distinguish between ``future'' and ``past'' directions in the spacetime.

Our main contribution is the formulation and analysis of the finite volume method based on monotone numerical fluxes and on a fully geometric approach to the discretization of \eqref{claw}.
For technical convenience and in order to be able to formulate the method as a time-stepping scheme, we assume that $M$ is foliated by compact slices. We built here on several earlier works. 
 The convergence of finite volume approximations of the initial value problem for hyperbolic conservation laws was established first in the Euclidean case by Cockburn, Coquel, and LeFloch \cite{CCL94,CCL95}
 and later for Riemannian manifolds in \cite{ABL}. Further generalizations were then provided in \cite{LeFlochOkutmustur2} and \cite{Gie09}.

We find here that in the finite volume discretisation of \eqref{claw} it is very natural to view the approximate solution $u^h$ as being defined via {\sl total flux functions} along faces of the triangulation, that is, the quantities denoted below by 
   $q_e := \int_e i^* \omega(u_e),$ where $e$ is a face and $i^*$ is the pull-back operator associated to the inclusion   $i: e \rightarrow M$. 
      For some faces which we call {\em spacelike}, we are able to go back and forth between the total flux $q_e$ and the ``physical value'' $u_e$.
One important contribution in the present paper is that we have eliminated the need to introduce (face size) averages $|e_K^+|$ which were used in LeFloch and Okutmustur's earlier formulation \cite{LeFlochOkutmustur2}.  In the present proposal, we consider it to be natural to use the flux $q_e$ for every face  $e$ while the quantity $u_e$ is introduced for spacelike faces only.

Interestingly enough, the formulation of the finite volume scheme requires {\sl further structure} on the manifold, especially an $n$-form along its boundary. This is necessary in order to determine an averaged value of the data within each boundary cell. It is not expected that stable schemes could be designed that would only require the fluxes across the boundary. 

We will prove several stability results for the proposed scheme with an emphasis on discrete versions of the entropy inequality. These stability estimates are sufficient to show that the sequence of approximate solutions created by the finite volume scheme converges to an {\em entropy measure valued solution} in the sense of DiPerna. By extending DiPerna's uniqueness theorem we show that the  sequence indeed converges to an {\em entropy solution}.
At the same time our analysis implies a natural extension of the $L^1$ contraction property satisfied by  hyperbolic conservation laws in Euclidean space, see \eqref{MTHM.1}.

An outline of this paper is as follows:
In Section 2 we introduce the notions of entropy weak/measure-valued solutions taking into account boundary data. Then, we discuss the concept of global hyperbolicity and gather our main results in Section 3. Section 4 is devoted to the presentation of the finite volume scheme and to the derivation of local stability estimates. In Section 5, we are in a position to prove global stability estimates for the finite volume scheme and to prove convergence of the approximate solutions toward the entropy solution of the Cauchy problem.

%=============================================================================================================
      
\section{Conservation laws posed on a spacetime}
\label{deux}
 
\subsection{Weak solutions}
 
In this preliminary section we present, in a slightly modified version, the formulation proposed in LeFloch and Okutmustur 
\cite{LeFlochOkutmustur2}. 
We assume that $M$ is an oriented, smooth $(n+1)$-manifold with smooth boundary $\del M$, which we refer to as 
a {\sl spacetime with boundary.} 
Given an $(n+1)$-form $\alpha$, its {\sl modulus} is defined as the $(n+1)$-form
\begin{equation}\label{def:abs}
|\alpha| : = |\alphab|\, dx^0 \wedge \cdots \wedge dx^n,
\end{equation}
where $\alpha = \alphab \,dx^0 \wedge \cdots \wedge dx^n$ is written in an oriented frame
determined from local coordinates $x=(x^\alpha)=(x^0, \ldots, x^n)$.
If $H$ is a hypersurface, we denote by $i=i_H : H \to M$ the canonical injection map, and
by $i^*=i_H^*$  the pull-back operator acting on differential forms defined on $M$.

We denote by $\CL^k(M)$ the space of all smooth fields of differential forms of degree $k \leq n+1$, 
and by $\DL^k(M) \subset \CL^k(M)$ the subset of compactly supported fields.

\begin{definition}
\label{123}
1. A {\rm flux field} $\omega$ on the $(n+1)$-manifold $M$
is a parameterized family $\omega(\ubar) \in \CL^n(M)$ of smooth fields
of differential forms of degree $n$,
that depends smoothly upon the real parameter $\ubar$.

2. The {\rm conservation law} associated with a flux field $\omega$ and with unknown $ u:  M \to \RR$ reads
\be
\label{LR.1}
d\big(\omega(u)\big)=0,
\ee
where $d$ denotes the exterior derivative operator and, therefore, $d\big(\omega(u)\big)$ is a field
of differential forms of degree $(n+1)$ on $M$.

3. A flux field $\omega= \omega(\ubar) \in \CL^n(M)$ is said to {\rm grow at most linearly}
if there exists some 
$\alpha \in C^\infty \Lambda^n (M) $
such that for every  
hypersurface $\mathcal{H}$ and $\bar u \in \RR$
\be
\label{111}
- i^*_\mathcal{H} \alpha \leq
i^*_\mathcal{H} \del_u \omega(\bar u) \leq 
i^*_\mathcal{H}\alpha.
\ee
\end{definition}

Given a {\sl smooth} solution $u$ of \eqref{LR.1} we can apply Stokes theorem on any open subset $S$
that is compactly included in $M$ and has piecewise smooth boundary $\del S$:
\be
\label{key64}
0 = \int_S d(\omega(u)) = \int_{\del S} i^*(\omega(u)).
\ee
There is a natural orientation on $S$ (as a subset of $M$) and, therefore, on $\del S$ such that \eqref{key64} indeed holds.
Similarly, given any smooth and compactly supported function $\psi \in \Dcal(\Mring) := \DL^0(\Mring)$,
 we can write
$$
d(\psi \, \omega(u)) = d\psi \wedge \omega(u) + \psi \, d(\omega(u)),
$$
where the differential $d\psi$ is a $1$-form field. Thus we find
$$
\int_M d(\psi \, \omega(u)) = \int_M d\psi \wedge \omega(u)
$$
and, by Stokes theorem,
\be
\label{LR.0}
\int_M d\psi\wedge \omega(u)=\int_{\del M}i^*(\psi\omega(u)).
\ee
This identity is satisfied by every smooth solution to \eqref{LR.1} and this motivates us to reformulate
\eqref{LR.1} in a {\it weak} form.

\begin{definition}
1. A function $u: M \to \RR$ is said to be {\rm locally integrable} (respectively {\rm integrable)} 
if it is measurable and for every non-negative, $(n+1)$-form field $\alpha \in \DL^{n+1}(\Mring)$
(resp.~$\alpha \in \CL^{n+1}(M)$)  
 one has 
\begin{equation}\label{def:l1}
\int_M |u| \, \alpha \quad \text{ is finite.} 
\end{equation}
The space of all such functions is denoted by $\Lloc(M)$ (resp.~$L^1(M)$).

2. Given a flux field  $\omega$ with at most linear growth, a function $u \in \Lloc(M)$ 
is called a {\rm weak solution} to the conservation law \eqref{LR.1} posed on $M$
if for every test-function $\psi \in \Dcal(\Mring)$ 
$$
\int_M d\psi\wedge \omega(u) = 0. 
$$
\end{definition}

The above definition makes sense, since \eqref{111} implies that
for every $1$-form field $\rho \in \DL^1(M)$ there exists a non-negative $(n+1)$-form field $\beta \in \DL^{n+1}(M)$ such that 
\be
\label{1111}
\sup_{\ubar \in \RR} \left| \rho \wedge \del_u \omega(\ubar) \right| \leq \beta,
\ee
so that under the conditions in Definition~\ref{123} the integral 
$\int_M d\psi\wedge \omega(u) $ is finite. 

The above definition
can be immediately generalized to functions defined on the boundary and we denote by 
$L^1(\del M)$ and $\Lloc(\del M)$ the space of all integrable (resp. locally integrable) functions defined on the boundary of the manifold. 
Observe that the {\sl $L^1$ norm} of functions in $\Lloc(M)$ or $L^1(M)$ is {\sl not} canonically defined,  since the value of the integral in \eqref{def:l1} depends on the choice of the form field $\alpha$.
On the other hand we will
also use the standard notation $L^1\Lambda^n(\Hcal)$ for the space of all integrable $n$-form fields defined on an 
($n$-dimensional) hypersurface $\Hcal$; it should be observed that when  $\Hcal$ is orientable
the $L^1$ norm of such a field {\sl is} uniquely defined,  since the integral of its absolute value is well-defined, confer \eqref{def:abs} for the definition of absolute value for differential forms.

%------------------------------------------------------------------------------------------------

\subsection{Entropy solutions} 

As is standard for nonlinear hyperbolic problems,
weak solutions must be further constrained by imposing initial, boundary, as well as entropy conditions, 
which we now discuss.

\begin{definition}
\label{key63}
A field of $n$-forms $\Omega=\Omega(\ubar) \in \CL^n(M)$ depending Lipschitz continuously on $\ubar$
is called a {\rm (convex) entropy flux field} for the conservation law \eqref{LR.1} if there exists
a (convex) Lipschitz continuous function $U: \RR \to \RR$ such that
$$
\Omega(\ubar) = \int_0^\ubar \del_u U (\vb) \, \del_u \omega(\vb) \, d\vb, \qquad \ubar \in \RR.
$$
It is said to be  {\rm admissible} if, moreover, $\sup | \del_u U | < \infty$, and the pair $(U,\Omega)$ is 
called an (admissible, convex) {\rm entropy pair.}
\end{definition}

For instance, if one picks up the family of functions $U(\ub, \vb):=|\ub - \vb|$, where $\vb$
is a real parameter, the corresponding family of entropy flux fields reads
\be
\label{KRZ}
\Omegabf(\ub, \vb) := \sgn(\vb - \ub) \, ( \omega(\vb) - \omega(\ub)),
\ee
which provides us with a natural generalization to spacetimes of Kruzkov's entropy pairs.

Based on the notion of entropy flux above, we can derive entropy inequalities in the following way.
Given any smooth solution $u$ and multiplying \eqref{LR.1} by $\del_u U(u)$
we obtain the balance law
$$
 d( \Omega(u) ) -(d\Omega)(u) + \del_u U(u) (d\omega)(u) 
 =\del_u U(u) d(\omega(u)) = 0.
$$
However, for discontinuous solutions this identity can not be satisfied as an equality and, instead,
we impose the {\it entropy inequalities}
\be
\label{LR.1i}
d( \Omega(u)) - (d\Omega)(u) + \del_u U (u) (d\omega)(u) \leq 0
\ee
 in the sense of distributions for all admissible entropy pairs $(U,\Omega)$.
These inequalities can be justified, for instance, via the vanishing viscosity method, that is
by searching for weak solutions that are realizable as limits of smooth solutions to a parabolic
regularization of \eqref{LR.1}.

To prescribe initial and boundary conditions, we observe that, without further assumptions
on the flux field (yet to be imposed shortly below), points along the boundary
$\del M$ can not be distinguished, and it is natural to prescribe the trace of the solution
along the {\sl whole} of the boundary $\del M$. This is possible provided
the boundary condition
\be
\label{LR.2i} 
u|_{\del M} = u_B
\ee
associated with some data $u_B: \del M \to \RR$,
is understood in a sufficiently {\sl weak} sense, as now defined.

\begin{definition}
\label{key92}
Let $\omega=\omega(\ub)$ be a flux field with at most linear growth and
let $u_B \in L^1(\del M)$ be a prescribed boundary function. 
Then, a function $u \in \Lloc(M)$  
   is called an {\rm entropy solution} to the boundary value problem 
   determined by the conservation law $\eqref{LR.1}$
and the boundary condition $\eqref{LR.2i}$
if there exists a field of $n$-forms $\gamma \in \Lloc\Lambda^n(\del M)$
such that
$$
\aligned
&\int_M \Big(  d \psi \wedge \Omega(u)
+ \psi \, (d  \Omega) (u) - \psi \, \del_u U(u) (d \omega) (u) \Big)
\\
& - \int_{\del M} \psi_{|\del M} \, \big(  i^*\Omega(u_B) + \del_u U(u_B)  \big(\gamma - i^*\omega(u_B) \big) \big)
\,  \geq 0
\endaligned
$$
for every admissible convex entropy pair $(U,\Omega)$ and every function $\psi \in \Dcal(M)$.
\end{definition}

In the above definition, all integrals under consideration are finite, in particular the one involving the entropy flux
since $u \in \Lloc(M)$ and any admissible entropy flux also satisfies the condition \eqref{1111}.

The above definition can be generalized to encompass solutions within the much larger class of
measure-valued mappings. Indeed, following DiPerna \cite{DiPerna}, we consider solutions that are no longer functions but
{\sl locally integrable Young measures,} i.e.,
weakly measurable maps $\nu: M \to \text{Prob}(\RR)$ taking values
within  the set of probability measures $\text{Prob}(\RR)$ and such that 
$$
\int_M \la \nu, |\cdot| \ra \, \alpha 
$$
is finite for every $(n+1)$-form field $\alpha \in \DL^{n+1}(\Mring)$. 

\begin{definition}
\label{LR.4}
Given a flux field $\omega=\omega(\ub)$ with at most linear growth and given
a boundary function $u_B \in L^1(\del M)$, a locally integrable Young measure $\nu: M \to \text{Prob}(\RR)$ is called 
an {\rm entropy measure-valued solution} to the boundary value problem \eqref{LR.1}, \eqref{LR.2i} 
if there exists a boundary field $\gamma \in \Lloc \Lambda^n(\del M)$
such that the entropy inequalities
$$
\aligned
& \int_M  \Big\la \nu,  d \psi \wedge \Omega(\cdot)  +  \psi \, \big((d  \Omega) (\cdot)
- \del_u U(\cdot) (d \omega) (\cdot)\big) \Big\ra
\\
& - \int_{\del M} \psi_{|\del M} 
\Big( i^*\Omega(u_B) + \del_u U(u_B)  \big(\gamma - i^*\omega(u_B)\big)\Big)  
\,  \geq 0
\endaligned
$$
are satisfied for every admissible convex entropy pair $(U,\Omega)$ and every function $\psi \in \Dcal(M)$.
\end{definition}

Finally, we introduce a geometric compatibility condition which is quite natural (and will simplify some of the follow-up statements), 
since it ensures that constants are trivial solutions to the conservation law ---This is a property shared by many models of fluid dynamics 
such as the shallow water equations on a curved manifold.

\begin{definition}
A flux field $\omega$ is called {\rm geometry-compatible} if it is closed for each value of the parameter,
\be
\label{LR.3}
( d \omega ) (\ubar)=0, \qquad \ubar \in \RR.
\ee
\end{definition}

When \eqref{LR.3} holds, then it follows from Definition~\ref{key63} that
every entropy flux field $\Omega$ also satisfies the geometric compatibility condition
\be
\label{dOmega}
(d\Omega) (\ubar) =0, \qquad \ubar \in \RR.
\ee
In turn, the entropy inequalities \eqref{LR.1i} satisfied by an entropy solution $u : M \to \RR$ simplify drastically,
and take the form
\be
\label{LR.1i-simple}
d( \Omega(u)) \leq 0.
\ee

The geometry compatibility condition \eqref{LR.3} ensures that constant states are trivial solutions of \eqref{claw}. In case of conservation laws on Riemannian manifolds, confer \eqref{Riemann}, this means nothing but
$\operatorname{div}_g f(\cdot, \bar u)=0$ for all $\bar u \in \RR.$  A large number of geometry compatible fluxes on the sphere can be found in \cite{BLM1}. 
The reader may also recall that the scalar problems considered here are simplified model problems for compressible fluid flows on space times, which naturally satisfy a version of the geometry compatibility condition, see
\cite{LM14} for details.

%============================================================================================================

\section{Well-posedness theory} 

\subsection{Global hyperbolicity}

In general relativity, it is a standard assumption that the spacetime should be globally hyperbolic.
This notion must be adapted to the present setting, since we do not have a Lorentzian structure, but solely the $n$-volume form structure associated with the flux field $\omega$.

From now on a flux field $\omega= \omega(\ubar)$ is fixed. 
The following definition imposes a non-degeneracy condition on that flux, which will be assumed from now.

\begin{definition}
\label{hyperb-def}
The flux field $\omega= \omega(\ubar)$ of the conservation law \eqref{LR.1} on the manifold $M$ is said to satisfy the
{\rm global hyperbolicity condition} if there exists a $1$-form field $T \in \CL^1(M)$ called a {\rm field of observers}
such that 
\be
\label{hyperb} T \wedge \del_u \omega(\ubar) > 0 \qquad \ubar \in \RR.  
\ee 
A hypersurface $\Hcal$ is called {\rm spacelike} if  
for every normal $1$-form field $N$ 
\be
\label{297}
N \wedge \del_u \omega(\ubar) \neq 0, \qquad \ubar \in \RR,
\ee
and one fixes the orientation on $\Hcal$ so that  
\be
\label{orie}
i^*_\Hcal \del_u \omega(\ubar) > 0, \qquad \ubar \in \RR. 
\ee
\end{definition}

To every hypersurface $\Hcal$ and entropy flux $\Omega$ we can associate the function 
$$ 
\aligned
q^\Omega_\Hcal : \quad & \RR \to \RR,
\qquad 
 \ubar \mapsto \int_\Hcal i^*\Omega(\ubar),  
\endaligned
$$
which represents the {\sl total entropy flux} along the hypersurface $\Hcal$. When $\Omega=\omega$, we also use the short-hand notation 
$q_\Hcal := q^\omega_\Hcal$.

\begin{lemma}
\label{convexities}

If $\Hcal$ is a spacelike hypersurface, then the total flux function $q_\Hcal$ is strictly monotone and, therefore,
one-to-one on its image. 

The orientation on $\Hcal$ implies the following positivity property of Kruzkov's entropy flux fields
\be
\label{333}
i^*_\Hcal \Omegabf(\ubar,\vbar) \geq 0, \qquad \ubar, \vbar \in \RR. 
\ee
Moreover, the total entropy flux functions $q^\Omega_\Hcal$ satisfy the identity 
\be
\label{444}
\del_q \left(q^\Omega_\Hcal \circ q_\Hcal^{-1} \right)
= 
\del_u U \circ q_\Hcal^{-1}.
\ee
\end{lemma}

\begin{proof} Since $\Hcal$ is spacelike, \eqref{297} holds. We introduce a vector field $X$ along $\Hcal$ (not tangential to the hypersurface) 
such that $\la N, X \ra >0$. Then, at each point of $\Hcal$ we can supplement this vector with $n$ vectors tangent to $\Hcal$  
so that $(X, e_1, \ldots, e_n)$ is a positively-oriented (say) basis. 
The main point is to make a continuous selection along $\Hcal$, the specific sign being irrelevant for our present argument. 
Hence, by \eqref{297} one has 
$(N \wedge \del_u \omega(u))(X, e_1, \ldots, e_n)$ keeps a constant sign for all $u$. Since $N$ is a normal form we 
have 
$$
(N \wedge \del_u \omega(u))(X, e_1, \ldots, e_n) = \la N, X \ra \, (\del_u \omega(u))(e_1, \ldots, e_n), 
$$
so that $(\del_u \omega(u))(e_1, \ldots, e_n)$ also keeps a constant sign for all $u$. This shows that the function
$
\ubar \mapsto \int_\Hcal i^*\del_u \omega(\ubar),  
$
never vanishes. Hence, the flux $q_\Hcal$ is a strictly monotone function. 

From the definitions, we can compute 
$$
i^*_\Hcal \Omegabf(\ubar, \vbar) = \sgn(\vbar - \ubar) \, i^*_\Hcal \int_\ubar^\vbar   \del_u \omega(\wbar) \, d\wbar \geq 0,
$$
which is \eqref{333}. 
On the other hand, the proof of the last identity \eqref{444} is obvious from the definitions. 
\end{proof}

%----------------------------------------------------------------------------------------------------------

\subsection{Existence and uniqueness results}
\label{wellp}

The given observer $T$ canonically determines the following notion of causality. 

\begin{definition}
Given two hypersurfaces $\Hcal, \Hcal' \subset M$ such that $\del \Hcal, \del \Hcal' \subset \del M$, one says that
$\Hcal'$ {\rm lies in the future of} $\Hcal$ and one writes $\Hcal \prec \Hcal'$
if there exists a smooth and one-to-one mapping
$F: \Hcal \times [0,1] \to M$ 
such that
\bel{eq:fhs}
\aligned
& F(\Hcal \times \{0\}) =\Hcal, \qquad F(\Hcal \times \{1\}) = \Hcal', 
\\
& \la T, DF(\del_s ) \ra >0, \qquad F(\del \Hcal \times [0,1]) \subset \del M,
\endaligned
\ee
where $DF$ denotes the tangent map and $\del_s = \del/\del s$ denotes the coordinate tangent vector corresponding to $[0,1]$.
The set $\Bcal(\Hcal,\Hcal'): = F(\del \Hcal \times [0,1]) \subset \del M$
is called the {\sl boundary of the region} limited by $\Hcal$ and $\Hcal'$. 
\end{definition}

To state how the boundary conditions are assumed we need the following definition.

\begin{definition}
\label{679}
Given any open subset $S \subset M$ with piecewise smooth boundary, 
a smooth $1$-form field $N$  defined on a smooth manifold $\Hcal \subset \del S$ is called {\rm outward pointing with respect to $S$} 
if $\la N, X \ra > 0$ on $\Hcal$ for all tangent vectors $X \in T_pM\setminus\{0\}$ associated with curves {\rm leaving} $S$.
One calls $\Hcal$ an {\rm outflow boundary of $S$} (respectively an {\rm inflow boundary of $S$})
if the hypersurface $\Hcal$ is spacelike and any outward pointing $1$-form field $N$ on $\Hcal$ satisfies the sign condition 
$N \wedge \del_u \omega(\ubar) > 0$ (resp.~$ N \wedge \del_u \omega(\ubar) < 0$) for all $\ubar$.
\end{definition}

In the following, we require that the inflow boundary of the spacetime $(\del M)^-$ is non-empty, as this ensures
that the boundary data will be assumed in a strong sense on an open subset of the boundary, at least. 
We will give an example  later (Example 1 in the appendix) which shows that $(\del M)^- \not= \emptyset$ need not be a consequence of the hyperbolicity condition. 

From now on we assume the existence of a foliation by spacelike hypersurfaces of $M$, i.e. 
\be
\label{foliation}
M= \bigcup_{0\leq t \leq T} \Hcal_t,
\ee
where each slice $\Hcal_t$ is a spacelike hypersurface and has the topology of a smooth $n$-manifold $N$ with boundary.
Furthermore we impose that $\Hcal_0$ is an inflow boundary of $M$.
By the foliation we have $M=[0,T] \times N$ topologically and a decomposition of the boundary is induced: 
\be 
\aligned
& \del M = \Hcal_0 \cup \Hcal_T \cup \del^0M,
\qquad
\quad
 \del^0M = \bigcup_{0\leq t \leq T} \del  \Hcal_t. 
\endaligned
\ee

The existence of a foliation by spacelike hypersurfaces is non-trivial, but it is satisfied in many examples of practical interest, e.g., fluid flows on Schwarzschild spacetimes, see \cite[Sec 2.1]{LM14}, for details.

Our main theory of existence, uniqueness, and stability is as follows. 
Observe that the  $L^1$ 
stability property is fully geometric in nature, in that it is stated for any two hypersurfaces such that one lies in the future of the other.

\begin{theorem}[Well-posedness theory for conservation laws on a spacetime]
\label{main}
Let $M$ be an $(n+1)$-dimensional spacetime with boundary and $\omega=\omega(\ubar)$ 
be a geometry-compatible flux field on $M$ growing at most linearly, 
and assume that the global hyperbolicity condition \eqref{hyperb} holds,  
the inflow boundary $(\del M)^-$ is non-empty and $M$ admits a foliation. 
Given boundary data $u_B \in L^\infty(\partial M)$  on $\del M$,
the boundary value problem determined by the conservation law \eqref{LR.1} and the boundary condition \eqref{LR.2i} 
admits
a unique entropy solution $u \in \Lloc(M)$ which, moreover, has well-defined $L^1$ traces on any
spacelike hypersurface.
These solutions determine a (Lipschitz continuous) contracting semi-group in the sense that 
for any two hypersurfaces $\Hcal \prec \Hcal'$ and any Kruzkov entropy $\Omegabf$
and 
any  $A_B \in L^\infty \Lambda^n (\del M)$ such that
\[ \big|\del_ u \omega |_{\del M}\big| \leq A_B \text{ in } \del M \times \mathbb{R}\]
the following inequality holds
\be
\label{MTHM.1}
\int_{\Hcal'} i_{\Hcal'}^* \Omegabf\big( u_{\Hcal'}, v_{\Hcal'} \big)
\leq \int_{\Hcal} i^*_\Hcal \Omegabf\big( u_\Hcal,v_\Hcal \big)
+ \int_{\Bcal} \big| u_B - v_B  \big| A_B, 
\ee
where $\Bcal:=\Bcal(\Hcal,\Hcal')$ is the boundary between $\Hcal$ and $\Hcal'$.
Moreover, the boundary data $u_B$ along any (inflow with respect to $M$) spacelike parts $\Hcal_B \subset \del M$ 
is assumed in the strong sense
\be
\label{MTHM.23}
\lim_{\Hcal \to \Hcal_B} \int_{\Hcal} i^*_{\Hcal} \Omegabf\big( u_{\Hcal}, v_{\Hcal} \big) 
= \int_{\Hcal_B} i^*_{\Hcal_B} \Omegabf(u_B,v_B),
\ee
where $\Hcal$ is a sequence of hypersurfaces approaching the boundary $\Hcal_B$ in a sufficiently strong topology.
\end{theorem}

More generally, one can also express the contraction property within an arbitrary open subset with smooth boundary: 
similarly to \eqref{MTHM.1}, the total flux over the {\sl outflow} part of the boundary of this subset 
is controlled by the total flux over the remaining part. 
We can also extend a result originally established by DiPerna \cite{DiPerna}
(for conservation laws posed on the Euclidean space)
within the broad class of entropy measure-valued solutions.

\begin{theorem}[Uniqueness of measure-valued solutions for conservation laws on spacetimes] 
\label{MTHM.2}
Let $\omega$ be a geometry-compatible flux field on a spacetime $M$ 
growing at most linearly and satisfying the global hyperbolicity condition \eqref{hyperb}.
Then, any locally integrable entropy measure-valued solution $\nu$ (see Definition \ref{LR.4}) 
to the initial value problem \eqref{LR.1}, \eqref{LR.2i} reduces to a Dirac mass at each point, more precisely
\be
\nu= \delta_u,
\ee
where $u \in \Lloc(M)$ is the unique entropy solution to the same problem.
\end{theorem}

%==============================================================================================================

\section{Finite volume scheme based on total flux functions}
\label{finit} 

\subsection{Triangulations and numerical flux functions}
 
To introduce the finite volume scheme, we first need to introduce a triangulation
(or, rather, a family of triangulations).

\begin{definition}
\begin{itemize}
\item
A {\rm triangulation} $\Tau$ of the spacetime $M$ is a set of disjoint open subsets $K \subset M$ called {\sl cells}
such that $\cup_{K \in \Tau} \Kbar = M$ and the boundary of each cell $K$ is the union of finitely many smooth hypersurfaces 
$e \subset M$, called the faces of $K$. The set of faces of $K$ is denoted $\del K$. 
For any two cells $K,K'$ one also requires that $\Kbar \cap \Kbar'$ is a common face of $K$ and $K'$ or a submanifold of $M$ 
with co-dimension $2$, at least. 

\item The triangulation is said to be {\sl admissible} 
if  every cell $K$ admits one inflow face $e_K^-$ and one outflow face $e_K^+$ (in the sense of Definition~\ref{679} 
and with respect to $K$). The set of remaining faces (which might also be inflow or outflow faces) is denoted 
by $\del^0 K:= \del K \setminus\{e^\pm_K\}$. In addition, 
one requires that every inflow face $e_K^-$ is the outflow $e_{K'}^+$ of some other cell $K'$ or else
a subset of $\Hcal_0$, and that 
every {\rm vertical face} $e^0_K \in \del^0K$ is also a vertical face of some other cell or else a part of the vertical boundary $\del^0 M$.

\item A triangulation is said to be {\rm associated with the foliation}  
if there exists a sequence of times $0=t_0 < t_1 < \dots < t_N=T$ such that all spacelike faces $e_K^-,e_K^+$ are subsets of the slices $\Hcal_n := \Hcal_{t_n}$ for some $n=0,\dots,N$ and determine triangulations of them.
\end{itemize}
\end{definition}

It will be convenient to fix orientations as follows. 
Each face $e \in \del^0 K$  is oriented such that an $n$-form  $\eta \in \CL^n(e)$ is positive if it satisfies 
  $N \wedge \eta >0$
  for every outward pointing $1$-form  $N$.
 These orientations together with the orientations of $e_K^\pm$ given in Definition \ref{hyperb-def} will yield the desired signs in the definition of the finite volume scheme as can be seen by the following observation: 
 For every smooth solution $u$ of 
\eqref{LR.1} Stokes theorem yields 
\be
\label{excons}
\int_{e_K^+}i^*\omega (u) = \int_{e^-_K}i^*\omega (u) - \sum_{e_K^0 \in \del K^0} \int_{e^0_K}i^*\omega (u).
\ee

%----------------------------------------------------------------------------------------------------------- 

The finite volume scheme, as we propose to define it here, relies on {\sl approximate total flux} along 
the spacelike faces $e_K^\pm$, that is if $u$ is a solution we replace the total flux by its average
for some constant state $u_{e_K^\pm}$: 
$$
\int_{e_K^\pm} i^* \omega (u) \approx q_{e_K^\pm}(u_{e_K^\pm}).  
$$ 
The evolution of these values is determined by the Stokes formula provided 
we prescribe the total flux along vertical faces $e_K^0 \in \del^0 K$. 
Precisely, for each cell $K$ and $e \in \del^0 K$ (oriented as above) we introduce a Lipschitz continuous numerical flux
$Q_{K,e} : \mathbb{R}^2 \longrightarrow \mathbb{R}$ satisfying the classical consistency, conservation and monotonicity properties
\begin{itemize}

 \item $Q_{K,e}(\ubar, \ubar )= q_e (\ubar)$.

 \item $Q_{K,e}(\ubar, \vbar) = - Q_{K_e, e}(\vbar, \ubar)$. 

 \item $\frac{\del}{\del \ubar} Q_{K,e}(\ubar, \vbar ) \geq 0, 
       \qquad \frac{\del}{\del \vbar} Q_{K,e}(\ubar, \vbar ) \leq 0$,
\end{itemize}
where $K_e$ denotes the cell sharing the face $e$ with $K$.
Note that in the right hand side of the first condition it is understood that $e$ is oriented as a boundary of $K$.

Now, the consequence \eqref{excons} of Stokes formula suggests to pose, for each cell $K$,   
\be 
\label{fvgm}
q_{e_K^+} (u_{e_K^+}) = q_{e^-_K} (u_{e_K^-}) - \sum_{e \in \del^0 K} Q_{K,e}(u_{e^-_K},u_{e^-_{K_e}}),
\ee
as long as $e$ is not a part of the boundary of the manifold. 
The set $ \del^0 K$, by definition, may also include boundary faces: for such faces, the element $K_e$ is not defined and, instead, 
we determine the corresponding state $u_{e^-_{K_e}}$ (still denoted with the same symbol, with now $K_e$ being a ``fictitious'' cell), 
as follows. Along the boundary, to handle {\sl non-inflow faces} $e^0 \subset \del M$ we fix a positive
$n$-form field $\alpha_B$ (once for all) and 
define 
\be
\label{boundarydata}
u_{e^0}:= \Big(\int_{e^0} \alpha_B\Big)^{-1} \int_{e^0} u_B \, \alpha_B  .
\ee
This definition is used even on the inflow parts of the boundary
(where we could also inverse the total flux function
so that the additional $n$-form field is not really necessary on inflow faces).   

Finally to guarantee the stability of the scheme we impose the following version of the {\sl CFL stability condition:}
for all $K \in \Tau$
\be
\label{CFL}
\sum_{e^0 \in \del^0 K} \left|\frac{\text{sup}_{u,v} \big( \del_u Q_{K,e^0}(u,v) - \del_v Q_{K,e^0}(u,v) \big)} {\text{inf}_{u} \del_u q_{e^+_K}(u)} \right|\leq \frac{1}{2}, 
\ee
in which the supremum is taken over range of $u_B$.

\begin{remark}[CFL condition]
 Let us try to explain the meaning of the quantities in \eqref{CFL} in case of a conservation law on a Riemannian manifold, confer \eqref{Riemann} and \eqref{fluxformR}.
 In this case the restriction of $\del_u \omega$ to the tangent space of the spatial face $e_K^+$ is nothing but $\sqrt{|g|} dx^1 \wedge \dots \wedge dx^n$, so that 
$ \text{inf}_{u} \del_u q_{e^+_K}(u)$ coincides with the volume of $e_K^+$ measured by the Riemannian metric.
Similarly, $Q_{K,e^0}$ mesures the flux accross the spacetime face $e^0$ whose size is the product of the time step $\Delta t$ and the area of the corresponding $(n-1)-$dimensional face of $e_K^+.$
Thus, for a non-degenerate triangulation with mesh-width $h$, the expression on the left in \eqref{CFL} scales as $\Delta t \cdot h^{n-1} \cdot h^{-n}$.
\end{remark}

%----------------------------------------------------------------------------------------

\subsection{Convex decomposition and local entropy inequalities}

Our analysis of the finite volume scheme is based on a convex decomposition of the fluxes. 
This technique goes back to Tadmor \cite{Tadmor} for one-dimensional problems, Coquel and LeFloch \cite{CLF1,CLF2} for equations in several space dimensions, and Cockburn, Coquel and LeFloch \cite{CCL95} for finite volume schemes. 
Due to the geometry compatibility of the flux \eqref{LR.3} and our choice of orientations of faces Stokes Theorem implies for each cell $K$
$$
 0= \int_K d(\omega(u_{e_K^-})) = q_{e_K^+}(u_{e_K^-}) -q_{e_K^-}(u_{e_K^-}) + \sum_{e^0 \in \del^0 K} Q_{K,e^0}(u_{e^-_K},u_{e^-_K}) .
$$
Subtracting this identity from \eqref{fvgm} yields
$$
 q_{e_K^+}(u_{e_K^+}) - q_{e_K^+}(u_{e_K^-}) + \sum_{e^0 \in \del^0 K} \left( Q_{K,e^0}(u_{e^-_K},u_{e^-_{K_{e^0}}}) - Q_{K,e^0}(u_{e^-_K},u_{e^-_K}) \right) =0.
$$
 
Suppose we are given some real $\lambda_{K,e^0} \geq 0$ for each vertical face $e^0 \in \del^0K$ with $K \in \Tau$ such that $\sum_{e^0 \in \del^0K} \lambda_{K,e^0}=1$
and define $\tilde u_{K,e^0} $ by
\be
\label{qtilde}
q_{e_K^+}(\tilde u_{K,e^0}) = q_{e_K^+}(u_{e_K^-}) -
\frac{1}{\lambda_{K,e^0}} \left( Q_{K,e^0}(u_{e^-_K},u_{e^-_{K_{e^0}}}) - Q_{K,e^0}(u_{e^-_K},u_{e^-_K}) \right),
\ee
where we will have to show that the right hand side lies in the image of $q_{e_K^+}$. So
we have a {\sl convex decomposition}
\bel{convdec} q_{e_K^+}(u_{e_K^+}) = \sum_{e^0 \in \del^0 K} \lambda_{K,e^0} q_{e_K^+}(\tilde u_{K,e^0}).\ee
For the choice of $\lambda_{K,e^0} \geq 0$, at least the ratio  
$$
\frac{1}{\lambda_{K,e^0}} \left( Q_{K,e^0}(u_{e^-_K},u_{e^-_{K_{e^0}}}) - Q_{K,e^0}(u_{e^-_K},u_{e^-_K}) \right)
$$ 
should be finite so that 
$ \lambda_{K,e^0} q_{e_K^+}(\tilde u_{K,e^0})$
is well-defined.
This is fulfilled with the following definition:
$$ \hat \lambda_K := \sum_{e^0 \in \del^0K} \underbrace{ \left|\frac{\text{sup}_{u,v} \left(\del_u Q_{K,e^0}(u,v) - \del_v Q_{K,e^0}(u,v)\right)} {\text{inf}_u \del_u q_{e_K^+}(u)} \right|}
      _{=:\hat \lambda_{K,e^0}} \leq \frac{1}{2}, \quad \lambda_{K,e^0} := \frac{\hat \lambda_{K,e^0}}{\hat \lambda_K}.
$$

We will show that the right hand side of \eqref{qtilde} lies in the image of $q_{e_K^+}$ for the case 
$u_{e_K^-} \leq u_{e_{K_{e^0}}^-}$, the other case follows analogously. We have
$$
\aligned
0 &\leq  -\frac{1}{\lambda_{K,e^0}} \left( Q_{K,e^0}(u_{e^-_K},u_{e^-_{K_{e^0}}}) - Q_{K,e^0}(u_{e^-_K},u_{e^-_K})\right)\\
  &\leq  \frac{1}{2} \frac{\text{inf}_u \del_u q_{e_K^+}(u)}{\text{sup}_{u,v}|\del_v Q_{K,e^0}(u,v)| }
          \text{sup}_{u,v}|\del_v Q_{K,e^0}(u,v)| (u_{e_{K_{e^0}}^-}-u_{e_K^-})\\
  &\leq  \frac{1}{2} (q_{e_K^+}(u_{e_{K_{e^0}}^-}) - q_{e_K^+}(u_{e_K^-}))
\endaligned
$$
due to the monotonicity of $q_{e_K^+}$. Hence the right hand side of \eqref{qtilde} lies in the interval $[q_{e_K^+}(u_{e_K^-}),q_{e_K^+}(u_{e_{K_{e^0}}^-})]$ and therefore in the image of $q_{e_K^+}$.
The convex decomposition \eqref{convdec} enables us to prove entropy inequalities, as follows.

\begin{lemma}\label{convdecflux}
For every convex entropy flux $\Omega$ and cell $K$ we have 
\bel{eq:convdecflux}
 q^\Omega_{e_K^+}(u_{e_K^+}) \leq \sum_{e^0 \in \del^0K} \lambda_{K,e^0} q^\Omega_{e_K^+}(\tu),
 \ee
which makes sense as $e_K^+$ is oriented. 
\end{lemma}

\begin{proof}
Using integration by parts we get
$$
\aligned
& \sum_{e^0 \in \del^0K} \len \left( i^* \Omega(\tu) - i^*\Omega(\uep) \right)\\
& = \sum_{e^0 \in \del^0K} \len \Big[\left[ \del_u U(v) i^*\omega(v)\right]_{\uep}^\tu + \int_{\tu}^\uep \del_{uu}U(v) i^*\omega(v)\, dv\Big]\\
& = \sum_{e^0 \in \del^0K} \len \Big[\del_u U(\uep) \left[ i^*\omega(\tu) - i^*\omega(\uep) \right]
%\\
%& \qquad 
+ \int_{\tu}^\uep \del_{uu}U(v)\left( i^*\omega(v)- i^* \omega (\tu) \right)\, dv\Big].
\endaligned
$$
When we integrate over $e_K^+$  the first term in the last line vanishes due to \eqref{convdec},  while the latter term is non-negative since $U$ is convex and 
$i^* \del_u \omega  $ is a positive $n$-form.
\end{proof}

The next step is the derivation of  an entropy inequality for the faces. Note that here and in the sequel we do not require entropies to be admissible.

\begin{lemma}[Entropy inequality for the faces]\label{num-ent-flux}
For every convex entropy pair $(U,\Omega)$ and each $K \in \Tau$ and $e^0 \in \del^0K$ there exists a numerical entropy flux
function $Q^\Omega_{K,e^0}: \mathbb{R}^2 \longrightarrow \mathbb{R}$ satisfying for every $\ubar, \vbar \in \mathbb{R}:$
\begin{itemize}
\item $Q^\Omega_{K,e^0}$ is consistent with the entropy flux $\Omega$:
\be
Q^\Omega_{K,e^0}(\ubar, \ubar)= \int_{e^0} i^*\Omega(\ubar).
\ee
\item Conservation 
\be\label{entrcons}
Q^\Omega_{K,e^0}(\ubar, \vbar)=-Q^\Omega_{K_{e^0},e^0}(\vbar, \ubar).
\ee
\item Discrete entropy inequality
\be\label{dei}
q^\Omega_{e_K^+}(\tu) \leq  q^\Omega_{e_K^+}(u_{e_K^-}) -
                         \frac{1}{\lambda_{K,e^0}} \left(Q^\Omega_{K,e^0}(u_{e_K^-},u_{e^-_{K_{e^0}}}) - Q^\Omega_{K,e^0}(u_{e_K^-},u_{e_K^-}) \right) .
\ee
\item Discrete boundary entropy inequality
\be\label{49b}
 q^\Omega_{e_K^+}(\bu) \leq q^\Omega_{e_K^+}(u_{e_{K_{e^0}}^-})  
 +  \frac{1}{\lambda_{K,e^0}} 
 \left(Q^\Omega_{K,e^0}(u_{e_K^-},u_{e^-_{K_{e^0}}}) -Q^\Omega_{K,e^0}(u_{e^-_{K_{e^0}}},u_{e^-_{K_{e^0}}}) \right) ,
\ee
where $\bu$ is defined by
\be \label{baru}
 q_{e_K^+}(\bu):=  q_{e_K^+}(u_{e_{K_{e^0}}^-}) + \frac{1}{\lambda_{K,e^0}} 
\left(Q_{K,e^0}(u_{e_K^-},u_{e^-_{K_{e^0}}}) -Q_{K,e^0}(u_{e^-_{K_{e^0}}},u_{e^-_{K_{e^0}}}) \right).
\ee
\end{itemize}
\end{lemma}

\begin{remark}
Similarly to the proof of the well-definedness of $\tu$ one can show that the right hand side of \eqref{baru} lies between $q_{e_K^+}(u_{e_K^-})$ and $q_{e_K^+}(u_{e_{K_{e^0}}^-})$. Hence $\bu$ is well-defined.
\end{remark}

\begin{proof}
Step 1:
For $u,v \in \RR,$ $K \in \Tau$, $e^0 \in \del^0K$ and $K' \in \{K,K_{e^0}\}$ we introduce the notation
$$
H_{K,K',e^0}(u,v):= q_{e_K^+}(u) - \frac{1}{\len} \left( Q_{K',e^0}(u,v) - Q_{K',e^0}(u,u) \right),
$$
where
$$
Q_{K_{e^0},e^0}(u,v):= -Q_{K,e^0}(v,u) \ \text{ for } e^0 \in \del^0M,
$$
and observe that
$$
H_{K,K',e^0} (u,u)= q_{e_K^+}(u).
$$
We claim that $H_{K,K',e^0}$ satisfies the following monotonicity properties:
\be
\label{H-mono}
\frac{\del}{\del u}H_{K,K',e^0}(u,v) \geq 0, \qquad \frac{\del}{\del v}H_{K,K',e^0}(u,v) \geq 0.
\ee
The second property is obvious because of the monotonicity properties of the numerical fluxes.
For the first property the monotonicity properties and the fact that $\del_u q_{e_K^+}$ is  positive imply
$$
\aligned
\frac{\del}{\del u}H_{K,K',e^0}(u,v) &\geq \del_u q_{e_K^+}(u) - \frac{1}{\len}| \del_u Q_{K',e^0}(u,v)| - \frac{1}{\len}| \del_2 Q_{K',e^0}(u,u)|\\
& \geq \left( 1 - 2\frac{\hat \lambda_{K,e^0}}{\len} \right) \del_u q_{e_K^+} \\
& \geq  \left( 1 - 2\hat \lambda_K \right) \del_u q_{e_K^+} ,
\endaligned
$$
which is non-negative due to the CFL condition \eqref{CFL}.

\

Step 2: We will establish the entropy inequalities only for the family of Kruzkov's entropies $\Omegabf$.
Therefore we introduce the numerical versions of Kruzkov's entropy fluxes
$$
\Qbf_{K,e^0}(u,v,c):= Q_{K,e^0} (u \vee c, v \vee c) - Q_{K,e^0} (u \wedge c, v \wedge c),
$$
where $a \vee b :=\max(a,b)$ and $a \wedge b := \min(a,b)$.
We observe that $\Qbf_{K,e^0}(u,v,c)$ satisfies the first two assertions of the lemma
when the entropy flux is replaced by Kruzkov's entropy flux $\Omega = \Omegabf$.
We observe that by definition of $H_{K,K',e^0}$
\be
\label{46}
\aligned
& H_{K,K',e^0}(u \vee c,v \vee c) - H_{K,K',e^0}(u \wedge c, v \wedge c)\\
& = q_{e_K^+}^{\bf \Omega}(u,c) - \frac{1}{\len} \left( \Qbf_{K',e^0}(u,v,c) - \Qbf_{K',e^0}(u,u,c) \right),
\endaligned
\ee
where we used
$$
q_{e_K^+}^{\bf \Omega}(u,c) := q_{e_K^+}(u \vee c) - q_{e_K^+}(u \wedge c) .
$$
Next we will check that for any $c \in \RR$
\be
\label{47a}
H_{K,K,e^0}(u_{e_K^-} \vee c,u_{e_{K_{e^0}}^-} \vee c) - H_{K,K,e^0}(u_{e_K^-} \wedge c,u_{e_{K_{e^0}}^-} \wedge c) \geq
q_{e_K^+}^{\bf \Omega}(\tu,c).
\ee

To this end, we first observe that due to the monotonicity of $H_{K,e^0}$
\be\label{47c}
\aligned
H_{K,K',e^0}(u ,v ) \vee H_{K,K',e^0}(c, c) &\leq H_{K,K',e^0}(u \vee c,v \vee c),\\
H_{K,K',e^0}(u ,v ) \wedge H_{K,K',e^0}(c, c) &\geq H_{K,K',e^0}(u \wedge c,v \wedge c).
\endaligned
\ee
Inserting $K'=K$, $u=u_{e_K^-}$ and $v=u_{e_{K_{e^0}}^-}$ in \eqref{47c} yields
$$
\aligned
& H_{K,K,e^0}(u_{e_K^-} \vee c,u_{e_{K_{e^0}}^-} \vee c) - H_{K,K,e^0}(u_{e_K^-} \wedge c,u_{e_{K_{e^0}}^-} \wedge c)\\
&\geq \left| H_{K,K,e^0}(u_{e_K^-} ,u_{e_{K_{e^0}}^-} ) - H_{K,K,e^0}(c,c) \right|\\
& \stackrel{\eqref{qtilde}}= |q_{e_K^+}(\tilde u_{K,e^0}) - q_{e_K^+}(c)|\\
& = \text{sgn}(\tu - c) (q_{e_K^+}(\tu) -  q_{e_K^+}(c)) = q_{e_K^+}^{\bf \Omega}(\tu,c).
\endaligned
$$
This proves \eqref{47a}. Combining \eqref{47a} with \eqref{46} (using $K'=K$, $u=u_{e_K^-}$ and $v=u_{e_{K_{e^0}}^-}$)  yields the third assertion of the lemma.

Let us now prove:
\be\label{47b}
H_{K,K_{e^0},e^0}(u_{e_{K_{e^0}}^-}\vee c,u_{e_K^-}  \vee c) - H_{K,K_{e^0},e^0}( u_{e_{K_{e^0}}^-}\wedge c,u_{e_K^-} \wedge c) \geq
q_{e_K^+}^{\bf \Omega}(\bu,c).
\ee
It follows by inserting $K'=K_{e^0}$, $u=u_{e_{K_{e^0}}^-}$ and $v=u_{e_K^-}$ in \eqref{47c} which yields
$$
\aligned 
&  H_{K,K_{e^0},e^0}( u_{e_{K_{e^0}}^-} \vee c,u_{e_K^-} \vee c) - H_{K,K_{e^0},e^0}(u_{e_{K_{e^0}}^-} \wedge c,u_{e_K^-} \wedge c)\\
& \geq | H_{K,K_{e^0},e^0}( u_{e_{K_{e^0}}^-},u_{e_K^-}) - H_{K,K_{e^0},e^0}(c,c)|\\
& \stackrel{\eqref{baru}}=|q_{e_K^+}(\bu) -  q_{e_K^+}(c) | = q_{e_K^+}^{\bf \Omega}(\bu,c).
\endaligned
$$
Inequality \eqref{47b} combined with \eqref{46} (setting $K'=K_{e^0}$, $u=u_{e_{K_{e^0}}^-}$ and $v=u_{e_K^-}$) and the conservation property of $Q_{K,e^0}^{\bf \Omega}$  yields the fourth assertion of the lemma.
\end{proof}

Combining the last two lemmas we get: 

\begin{lemma}[Entropy inequality per cell] For each cell $K \in \Tau$ we have
$$
q^\Omega_{e_K^+}(u_{e_K^+}) - q^\Omega_{e_K^+}(u_{e_K^-}) + 
                         \sum_{e^0\in \del^0K} \left(Q^\Omega_{K,e^0}(u_{e_K^-},u_{e^-_{K_{e^0}}}) - Q^\Omega_{K,e^0}(u_{e_K^-},u_{e_K^-}) \right) \leq 0.
$$
\end{lemma}

%--------------------------------------------------------------------------------------------------------------------------

Finally, we turn our discussion to the boundary of $M$. 

\begin{lemma}[Discrete boundary condition]\label{dbc}
For all convex entropy pairs $(U,\Omega)$ the following discrete boundary condition holds for each $e^0$: 
\begin{eqnarray*}
&& Q^\Omega_{K,e^0}\left({u_{e_K^-}},u_{e_{K_{e^0}}^-}\right) - Q^\Omega_{K,e^0}\left(u_{e_{K_{e^0}}^-},u_{e_{K_{e^0}}^-}\right)\\
&\geq& \del_uU(u_{e_{K_{e^0}}^-}) \left[ Q_{K,e^0} (u_{e_K^-},u_{e_{K_{e^0}}^-}) - Q_{K,e^0}(u_{e_{K_{e^0}}^-},u_{e_{K_{e^0}}^-})\right]. 
\end{eqnarray*}
\end{lemma}

\begin{proof}
According to Lemma \ref{convexities} the function $q^\Omega_{e_K^+} \circ q^{-1}_{e_K^+}$ is convex so we have using \eqref{baru}
$$
\aligned
& q^\Omega_{e_K^+}(\bu) - q^\Omega_{e_K^+} (u_{e_{K_{e^0}}^-})\\
& \geq \left( q^\Omega_{e_K^+} \circ q^{-1}_{e_K^+} \right)' (q_{e_K^+}(u_{e_{K_{e^0}}^-})) \left( q_{e_K^+}(\bu) - q_{e_K^+} (u_{e_{K_{e^0}}^-})\right)\\
& = \frac{1}{\len} \left( q^\Omega_{e_K^+} \circ q^{-1}_{e_K^+} \right)' (q_{e_K^+}(u_{e_{K_{e^0}}^-}))
                                              \left( Q_{K,e^0}(u_{e_K^-},u_{e^-_{K_{e^0}}}) -Q_{K,e^0}(u_{e^-_{K_{e^0}}},u_{e^-_{K_{e^0}}}) \right)\\
& = \frac{1}{\len} (\del_u U) (u_{e_{K_{e^0}}^-}) \left( Q_{K,e^0}(u_{e_K^-},u_{e^-_{K_{e^0}}}) -Q_{K,e^0}(u_{e^-_{K_{e^0}}},u_{e^-_{K_{e^0}}}) \right).
\endaligned
$$
Combining this inequality with \eqref{49b} yields the assertion of the lemma.
\end{proof} 

%=============================================================================================================================

\section{Global estimates and convergence analysis}

\subsection{Global entropy bounds and inequalities}

We now establish a global bound on the ``discrete derivatives'' of the approximate solutions.
To state it we introduce the following notation:

\be
\aligned
\del \Tau^0 &= \{ e^0 \in \del^0K | K \in \Tau,  e^0 \subset \del^0 M\},\quad
&&&\Tau_j &= \{ K \in \Tau |  e_K^- \subset \Hcal_j\},\\
\del \Tau^0_j &= \{ e^0 \in \del^0K | K \in \Tau_j,  e^0 \subset \del^0 M\},\quad 
&&&\del \Tau_j &= \{ e^0 \in \del^0 K | K \in \Tau_j\}.
\endaligned
\ee

\begin{lemma}[Global entropy dissipation estimate]\label{gede}
The following estimate for the entropy dissipation holds
\be
\aligned  & \sum_{K \in \Tau_j} q_{e_K^+}^\Omega(u_{e_K^+}) + c \sum_{e^0 \in \del \Tau_j }\lambda_{K,e^0}  
\frac{(\inf_u \del_u q_{e_K^+})^2}{\sup_u \del_u q_{e_K^+}} | \tilde u_{K,e^0} - u_{e_K^+} |^2\\
\leq& - \sum_{e^0 \in \del \Tau^0_j} Q^\Omega_{K,e^0} (u_{e_K^-}, u_{e_{K_{e^0}}^-}) + \sum_{ K \in \Tau_j } q_{e_K^-}^\Omega(u_{e_K^-}) ,
\endaligned
\ee
where $2c$ is a modulus of convexity of $U$ and $\Omega$ is the associated entropy flux field.
When $\Omega$ is the $n$-form entropy flux field associated to the  entropy $U(u)=u^2,$ we have
\be 
\aligned& \sum_{e^0 \in \del \Tau_j }\lambda_{K,e^0} \frac{(\inf_u \del_u q_{e_K^+})^2}{\sup_u \del_u q_{e_K^+}}  | \tilde u_{K,e^0} - u_{e_K^+} |^2\\
\leq& - \sum_{e^0 \in \del \Tau_j^0} Q^\Omega_{K,e^0} (u_{e_K^-}, u_{e_{K_{e^0}}^-}) + \sum_{K \in \Tau_j} q_{e_K^-}^\Omega(u_{e_K^-}) .
\endaligned
\ee
\end{lemma}

\begin{proof}
We multiply \eqref{dei} by $\lambda_{K,e^0}$ and sum over all $e^0_K \in \del^0K$ and all $K \in \Tau_j$.
Due to \eqref{entrcons} this yields
\be \label{511}
\aligned
&\sum_{e^0 \in \del \Tau_j } \lambda_{K,e^0} q^\Omega_{e_K^+}(\tu) - \sum_{K \in \Tau_j}q^\Omega_{e_K^+} (u_{e_K^-})\\
+& \sum_{e^0 \in \del \Tau^0_j} Q^\Omega_{K,e^0}(u_{e_K^-},u_{e^-_{K_{e^0}}}) - \sum_{e^0 \in \del \Tau_j } Q^\Omega_{K,e^0}(u_{e_K^-},u_{e_K^-}) \leq 0. 
\endaligned
\ee
Applying $q^\Omega_{e_K^+} \circ q_{e_K^+}^{-1}$  to the convex decomposition \eqref{convdec} yields
\be\label{512}
\aligned
& q_{e_K^+}^\Omega(u_{e_K^+}) + \underbrace{\beta \sum_{e^0 \in \del^0K} \len | q_{e_K^+}(\tu) - q_{e_K^+}(u_{e_K^+}) | ^2}_{=:A}
& \leq \sum_{e^0 \in \del^0K} \len q_{e_K^+}^\Omega(\tu),
\endaligned
\ee
where $2\beta$ is a modulus of convexity of $q^\Omega_{e_K^+} \circ q_{e_K^+}^{-1}$.
Using  Lemma \ref{convexities} we have on the one hand
$$
\del_{qq} (q^\Omega_{e_K^+} \circ q_{e_K^+}^{-1}) = (\del_{uu} U \circ q_{e_K^+}^{-1} ) \cdot \left( \frac{1}{\del_u q_{e_K^+} \circ q_{e_K^+}^{-1}} \right),
$$
which implies 
$
\beta \geq \frac{c}{\sup_u \del_u q_{e_K^+}}.
$
On the other hand 
$$
| q_{e_K^+}(\tu) - q_{e_K^+}(u_{e_K^+}) | ^2 \geq ( \inf_u \del_u q_{e_K^+})^2 | \tu - u_{e_K^+} |^2.
$$
So we have
\begin{equation}\label{newlabel}
A \geq c\sum_{e^0 \in \del^0K} \len \frac{(\inf_u \del_u q_{e_K^+})^2}{\sup_u \del_u q_{e_K^+}} | \tu - u_{e_K^+} |^2.
\end{equation}
Combining \eqref{511}, \eqref{512}, and \eqref{newlabel} we get
\begin{equation*}
\aligned
& \sum_{K \in \Tau_j} q_{e^+_K}^\Omega (u_{e_K^+}) + 
c \sum_{e^0 \in \del \Tau_j} \lambda_{K,e^0} \frac{(\inf_u \del_u q_{e_K^+})^2}{\sup_u \del_u q_{e_K^+}}|\tu - u_{e_K^+}|^2 
- \sum_{K \in \Tau_j}q_{e_K^+}^\Omega(u_{e_K^-})\\
&\leq -\sum_{e^0 \in \del \Tau^0_j} Q^\Omega_{K,e^0}(u_{e_K^-},u_{e^-_{K_{e^0}}}) 
+ \sum_{e^0 \in \del \Tau_j} Q^\Omega_{K,e^0}(u_{e_K^-},u_{e_K^-}).
\endaligned
\end{equation*}
Because of \eqref{dOmega} Stokes Theorem implies for every $K \in \Tau$
$$
-q_{e^+_K}^\Omega (u_{e_K^-})
 + q_{e^-_K}^\Omega (u_{e_K^-}) = \sum_{e^0 \in \del^0K} Q^\Omega_{K,e^0}(u_{e_K^-},u_{e_K^-}).
 $$
Inserting this in the last inequality yields
\begin{eqnarray}\label{ged1}
&& \sum_{K \in \Tau_j} q_{e^+_K}^\Omega (u_{e_K^+}) 
+ c \sum_{e^0 \in \del \Tau_j} \lambda_{K,e^0} \frac{(\inf_u \del_u q_{e_K^+})^2}{\sup_u \del_u q_{e_K^+}}|\tu - u_{e_K^+}|^2 
\\
&\leq& -\sum_{e^0\in \del \Tau_j^0} Q^\Omega_{K,e^0}(u_{e_K^-},u_{e^-_{K_{e^0}}}) + \sum_{K \in \Tau_j} q_{e^-_K}^\Omega (u_{e_K^-}). \nonumber
\end{eqnarray} 
This proves the first assertion of the Lemma.
Noting that for $U(u)=u^2$ we have
$$
 q_{e^+_K}^\Omega (u_{e_K^+}) \geq 0
 $$
yields the second assertion.
\end{proof}

Let us choose a finite number of charts covering $M$ (once and for all, independent of the triangulation) in such a way that (for any sufficiently fine triangulation) each element of our triangulation is contained in one chart domain.
These charts induce $n$-forms $\alpha_{e^0}$ corresponding to Hausdorff measures on the faces $e^0$. 
Then we define for any $\psi \in \Dcal(M)$ and cell $K\in \Tau$ with vertical face $e^0\in \del^0K$
$$
\psi_{e^0} := \frac{\int_{e^0} \psi \alpha_{e^0}}{\int_{e^0}\alpha_{e^0}}, \quad \psi_{\del^0K}:= \sum_{e^0 \in \del^0K} \len \psi_{e^0}.
$$

%--------------------------------------------------------------------------------------------------------------------

\begin{lemma}[Global entropy inequalities]\label{52}
Let $\Omega$ be a convex entropy flux field and $\psi \in \Dcal(M)$ a non-negative test-function supported in $M \setminus \Hcal_T$. 
Then the finite volume approximation satisfies the following global entropy inequality
$$
\aligned
&- \sum_{K \in \Tau} \int_K d(\psi \Omega)(u_{e_K^-}) - \sum_{K \in \Tau_0} \int_{e_K^-} \psi i^*\Omega(u_{e_K^-})
  + \sum_{e^0 \in \del \Tau^0}  \psi_{e^0} Q^\Omega_{K,e^0}(u_{e_K^-},u_{e_{K_{e^0}}^-})\\
& \leq A(\psi) + B(\psi) + C(\psi)+ D(\psi) + E(\psi),
\endaligned
$$
with 
$$
\aligned
&A(\psi):= \sum_{K \in \Tau, e^0 \in \del^0K} \len (\psi_{\del^0K}-\psi_{e^0}) \left( q_{e_K^+}^\Omega(\tu)-q_{e_K^+}^\Omega(u_{e_K^+})\right),\\
&B(\psi):= \sum_{K \in \Tau, e^0 \in \del^0K} \int_{e^0} (\psi_{e^0}-\psi) i^*\Omega(u_{e_K^-}),\\
&C(\psi):= - \sum_{K \in  \Tau, e^0 \in \del^0K} \int_{e_K^+}\len (\psi_{\del^0K} - \psi)\left( i^*\Omega(\tu) - i^*\Omega(u_{e_K^+})\right),\\
&D(\psi):= -\sum_{K \in \Tau,e^0 \in \del^0K} \int_{e_K^+} \lambda_{K,e^0} \psi \del_u U(u_{e_K^+}) ( i^* \omega (\tu) - i^*\omega(u_{e_K^+})),\\
&E(\psi):= - \sum_{K \in  \Tau} \int_{e_K^+} (\psi_{\del^0K} - \psi)\left( i^*\Omega(u_{e_K^+}) - i^*\Omega(u_{e_K^-})\right).
\endaligned
$$
\end{lemma}

\begin{proof}
 By Stokes theorem we have for each cell $K$
\bel{521}
 - \int_K d(\psi \Omega)(u_{e_K^-}) + \int_{e_K^+} \psi i^* \Omega(u_{e_K^-}) - \int_{e_K^-} \psi i^* \Omega(u_{e_K^-})
 +\sum_{e^0 \in \del^0K} \int_{e^0} \psi i^* \Omega(u_{e_K^-}) =0.
\ee
Multiplying \eqref{dei} by $\psi_{e^0} \len$ and summing over $e^0 \in \del^0K$ we have for each cell $K$
\bel{522}
 \aligned
 & \sum_{e^0 \in \del^0K} \len \psi_{e^0} q_{e_K^+}^\Omega(\tu) - \sum_{e^0 \in \del^0K} \len \psi_{e^0} q_{e_K^+}^\Omega(u_{e_K^-})&\\
 & +\sum_{e^0 \in \del^0K}  \psi_{e^0} \left(Q^\Omega_{K,e^0}(u_{e_K^-},u_{e_{K_{e^0}}^-}) - Q^\Omega_{K,e^0}(u_{e_K^-},u_{e_K^-})\right)& \leq 0.
\endaligned
\ee
Furthermore, we have 
\bel{523}
  \psi_{\del^0K} q^\Omega_{e_K^+}(u_{e_K^-})- 
  \sum_{e^0 \in \del^0K }  \psi_{\del^0K} \len q^\Omega_{e_K^+}(u_{e_K^-})= 0,
\ee
where we used $\sum_{e^0 \in \del^0K} \len=1$.
Adding \eqref{521}, \eqref{522} and \eqref{523} and summing over all $K \in \Tau$ yields
\bel{524}
 \aligned
&- \sum_{K \in \Tau} \int_K d(\psi \Omega)(u_{e_K^-}) + \sum_{K \in \Tau} \int_{e_K^+} \psi i^*\Omega(u_{e_K^-})
- \sum_{K \in \Tau} \int_{e_K^-} \psi i^*\Omega(u_{e_K^-})&\\
&-\sum_{K \in \Tau, e^0 \in \del^0K} \len (\psi_{\del^0K}-\psi_{e^0}) \left( q_{e_K^+}^\Omega(\tu)-q_{e_K^+}^\Omega(u_{e_K^-})\right)&\\
& -\sum_{K \in \Tau, e^0 \in \del^0K} \int_{e^0} (\psi_{e^0}-\psi) i^*\Omega(u_{e_K^-}) + \sum_{e^0 \in \del \Tau^0}  \psi_{e^0} Q^\Omega_{K,e^0}(u_{e_K^-},u_{e_{K_{e^0}}^-})&\\
&+ \sum_{K \in \Tau} \psi_{\del^0K} \left(\sum_{e^0 \in \del^0K}\len q_{e_K^+}^\Omega(\tu) - q_{e_K^+}^\Omega(u_{e_K^-}) \right) \leq 0,
\endaligned
\ee
where we used
$
\sum_{K \in \Tau, e^0 \in \del^0K} \psi_{e^0} Q_{K,e^0}^\Omega (u_{e_K^-},u_{e_{K_{e^0}}^-}) 
= \sum_{e^0 \in \del \Tau^0}  \psi_{e^0} Q^\Omega_{K,e^0}(u_{e_K^-},u_{e_{K_{e^0}}^-})
$
due to the conservation property \eqref{entrcons}.
We observe that 
\be\label{525}
\aligned
\sum_{K \in \Tau} \int_{e_K^-} \psi i^*\Omega(u_{e_K^-}) 
= 
&\sum_{K \in \Tau}\int_{e_K^+} \psi i^*\Omega(u_{e_K^+})
   +\sum_{K \in \Tau_0} \int_{e_K^-} \psi i^*\Omega(u_{e_K^-})\\
   \leq & \sum_{K \in \Tau, e^0 \in \del^0K }\int_{e_K^+}\len \psi i^*\Omega(\tu)
   +\sum_{K \in \Tau_0} \int_{e_K^-} \psi i^*\Omega(u_{e_K^-})\\
   &- \sum_{K \in \Tau, e^0 \in \del^0 K} \int_{e_K^+} \lambda_{K, e^0} \psi \del_u U(u_{e_K^+}) ( i^* \omega (\tu) - i^*\omega(u_{e_K^+}))
\endaligned
\ee
similar to the derivation of \eqref{eq:convdecflux}.
Finally inserting \eqref{525} in \eqref{524} yields
\be\label{526}
\aligned
&- \sum_{K \in \Tau} \int_K d(\psi \Omega)(u_{e_K^-}) - \sum_{K \in \Tau_0} \int_{e_K^-} \psi i^*\Omega(u_{e_K^-})
\\
& + \sum_{e^0 \in \del \Tau^0}  \psi_{e^0} Q^\Omega_{K,e^0}(u_{e_K^-},u_{e_{K_{e^0}}^-})\\
& + \sum_{ K \in \Tau, e^0 \in \del^0K} \int_{e_K^+}\len (\psi_{\del^0 K} - \psi) \left( i^* \Omega (\tu) - i^*\Omega(u_{e_K^-}) \right)\\
&+ \sum_{K \in \Tau,e^0 \in \del^0K} \int_{e_K^+} \lambda_{K,e^0} \psi \del_u U(u_{e_K^+}) ( i^* \omega (\tu) - i^*\omega(u_{e_K^+}))
 \leq A(\psi) + B(\psi),
\endaligned
\ee
where we have used that
\begin{equation}
A(\psi)= \sum_{K \in \Tau, e^0 \in \del^0K} \len (\psi_{\del^0K}-\psi_{e^0}) \left( q_{e_K^+}^\Omega(\tu)-q_{e_K^+}^\Omega(u_{e_K^-})\right)
\ee
since 
for all $K \in \Tau$
$$
\sum_{e^0 \in \del^0K} \len (\psi_{\del^0K}- \psi_{e^0}) q^\Omega_{e_K^+} (u_{e_K^-}) =\sum_{e^0 \in \del^0K} \len (\psi_{\del^0K}- \psi_{e^0}) q^\Omega_{e_K^+} (u_{e_K^+}) =0.
$$ 
Using the definitions of $C(\psi), D(\psi), E(\psi)$ the assertion of the lemma follows from \eqref{526}.
\end{proof}

An easy consequence of Lemmas \ref{52} and \ref{dbc} is as follows. 

\begin{lemma}
\label{53}
Let $(U,\Omega)$ be an admissible, convex entropy pair and $\psi \in \Dcal(M)$ a non-negative test-function compactly supported in $M \setminus \Hcal_T$. 
Then the finite volume approximation satisfies the following global entropy inequality

\bel{53:eq}
\aligned
&- \sum_{K \in \Tau} \int_K d(\psi \Omega)(u_{e_K^-}) - \sum_{K \in \Tau_0} \int_{e_K^-} \psi i^*\Omega(u_{e_K^-})
+ \sum_{e^0 \in \del \Tau^{0}} \int_{e^0} \psi_{e^0} i^*\Omega(u_{e_{K_{e^0}}^-}) \\
&+ \sum_{e^0 \in \del \Tau^{0}}\psi_{e^0} \del_uU(u_{e_{K_{e^0}}^-}) \left( Q_{K,e^0}(u_{e_K^-},u_{e_{K_{e^0}}^-}) - Q_{K,e^0}(u_{e_{K_{e^0}}^-},u_{e_{K_{e^0}}^-}) \right)\\
& \leq A(\psi) + B(\psi) + C(\psi) + D(\psi) + E(\psi), 
\endaligned
\ee
where $A(\psi),B(\psi),C(\psi),D(\psi), E(\psi)$ were defined in Lemma \ref{52}.
\end{lemma}

%--------------------------------------------------------------------------------------------------------------

\subsection{Convergence analysis}

Recall that we fixed a positive $n$-form field 
$\alpha_B$ along the boundary $\del M$ in \eqref{boundarydata}. We emphasize that such a structure is necessary for the analysis of the convergence of 
the finite volume scheme, only, and was not required in the well-posedness theory. 

We assume that there exists a constant $c>0$ such that we have for every compact subset $D$ of a spacelike hypersurface $\Hcal$
\be\label{52ii} \frac{\sup_u \del_u q_{D}(u)}{\inf_u \del_u q_{D}(u)} \leq c.
\ee

For the convergence analysis we need to fix a metric $d$ on $M$. Then let $\{\Tau^h\}_{h > 0}$ be a family of triangulations of $M$
which are admissible with respect to the foliation $\{\Hcal_t\}_{t \in [0,T]}$ satisfying the following conditions:
There is a constant $c>0$ such that
\be \label{tri1}
\aligned
& d(x,y) \leq c h \quad \text{ for all }  x,y \in K, K \in \Tau^h,\\ 
& \frac{h^n}{c} < \inf_u \del_u q_{e_K^+}(u) \leq \sup_u \del_u q_{e_K^+}(u) < c h^n \quad 
\text{ for all } K \in \Tau^h,\\
& \int_{e^0} \alpha_B < ch^n \quad \text{ for all }
 e^0 \in \del^0K, K \in \Tau^h, e^0 \subset \del^0M,
\\
& \# \{ e^0 \in \del^0 K\} < c \quad \text{ for all }  K \in \Tau^h,
\\
& \# \{ e^0 \in \del^0 K: K \in \Tau^h_j, e^0 \subset \del^0 M\} < c h^{1 -n} \quad \text{ for all }  j \in \mathbb{N},
\endaligned
\ee 
where for any finite set $L$ we denote the number of elements of $L$ by $\#L$.
We also impose that for each compact subset $D \subset M$ there is some constant $c(D)>0$ such that
\be\label{tri2}
\aligned
& \#\{ K \in \Tau^h_j : K \cap D \not= \emptyset\} \leq c(D) h^{-n} \quad \text{ for all } j \in \mathbb{N},
\\
& \#\{ j \in \mathbb{N}: \exists K \in \Tau^h_j :  K \cap D \not= \emptyset\} \leq c(D) h^{-1} .\\ 
\endaligned
\ee

In addition, for any admissible entropy field $\Omega ,$ any $\bar u \in \mathbb{R}$ and any face $e^0 \in\del  \Tau^h$ we
define $\phi^\Omega_{e^0}(\bar u): e^0 \rightarrow \mathbb{R}$ via
\[ i^*_{e^0} \Omega (\bar u) = \phi^\Omega_{e^0}(\bar u) \alpha_{e^0}.\]
We also define
$\bar \phi^\Omega_{e^0}(\bar u)$ as the mean value of $\phi^\Omega_{e^0}(\bar u)$ with respect to $\alpha_{e^0}$.

We impose the following two additional conditions on the triangulation:
\begin{enumerate}
 \item Uniformly in $\bar u$ for all $K \in \Tau^h, \, e^0 \in \del^0 K$ 
\bel{curvcond}
  \frac{  \int_{e^0} | \phi^\Omega_{e^0}(\bar u) - \bar \phi^\Omega_{e^0}(\bar u)| \alpha_{e^0}} { \int_{e^0}  \alpha_{e^0}} = o(1) \quad \text{ for } h \rightarrow 0.
\ee

\item For every $j \geq 1$, every $\psi \in \mathcal{D}(M)$, every $\bar u \in \mathbb{R}$ and every admissible entropy field $\Omega$
\be\label{trichange}
\sum_{K \in  \Tau_j^h} \Big| \int_{e_K^-}(\psi_{\del^0 K^-} - \psi) i^*\Omega (u_{e_K^-})\\
 -  \int_{e_K^+} (\psi_{\del^0 K} - \psi) i^* \Omega (u_{e_K^-})\Big| = o(h), 
 \ee
 where for $K \in \Tau_j$ we denote by $K^-$ the cell having $e_K^-$ as outflow face.

\end{enumerate}

\begin{remark}\label{rem:triconds}
 \begin{enumerate}
 \item Note that \eqref{curvcond} is not a trivial consequence of $\Omega(\bar u)$ being a smooth field of $n$-forms.
 Derivatives of $\phi^\Omega_{e^0}$ do not only depend on $\Omega$ since the regularity of $e^0$ (i.e., curvature in some compatible Riemannian metric) enters via $i^*_{e^0}$.
  \item For hyperbolic conservation laws on Riemannian manifolds, i.e. \eqref{Riemann}, assumption \eqref{curvcond} is satisfied provided the curvatures of the faces of the (spatial) cells are uniformly bounded under
  mesh refinement.
  \item A careful study of the estimate of $B^h(\psi)$ in Lemma  \ref{54} reveals that we may slightly weaken assumption \eqref{curvcond}, e.g., 
it is not a problem if it is violated for some faces $e^0$ as long as their number inside any compact $D \subset M$ is bounded by $c(D) h^{-n}$.
 This happens, e.g. in case latitude-longitude grids are considered for conservation laws on the sphere.
\item Condition \eqref{trichange} restricts temporal changes (of the geometric properties) of the triangulation.
For hyperbolic conservation laws on manifolds with Lipschitz continuously evolving Riemannian metric it is automatically satisfied, if the triangulation does not change in time.
The change of its geometric properties due to the evolving Riemannian metric is compatible with \eqref{trichange}.
 \end{enumerate}
\end{remark}

\begin{lemma}\label{54}
 Let $(U,\Omega)$ be an admissible entropy pair with $U(0)=0$ and $\psi \in \mathcal{D}(M)$ a non-negative test function. Let $\{\Tau^h\}_{h > 0}$ be a family of triangulations of $M$ satisfying \eqref{tri1} and \eqref{tri2}. Provided $u_B \in L^\infty(M)$ we have
$A^h(\psi),B^h(\psi),C^h(\psi),D^h(\psi),E^h(\psi) \rightarrow 0 \ \text{for } h \rightarrow 0,$
for the functions $A(\psi),B(\psi),$ $C(\psi),D(\psi),E(\psi)$ defined in Lemma \ref{52}.
\end{lemma}

\begin{proof} 
We start by showing $A^h(\psi) \rightarrow 0$ for $h \rightarrow 0$.
Note that
\be\label{531} 
|A^h(\psi)|
  \leq \sum_{K \in \Tau^h, e^0 \in \del^0K} \len |\psi_{\del^0K}-\psi_{e^0}| \sup_u | \del_u q_{e_K^+}^\Omega| | \tu - u_{e_K^+}|.
\ee
We have due to \eqref{52ii}
$$
|\del_u q_{e_K^+}^\Omega (u)| = \left| \int_{e_K^+}\del_u U(u) i^* \del_u \omega \right| \leq c \| \del_u U\|_\infty \inf_u \del_u q_{e_K^+}(u),
$$ 
and hence
\begin{multline}
|A^h(\psi)| \leq 
\left(C\sum_{K \in \Tau^h, e^0 \in \del^0K} \len \inf_u \del_u q_{e_K^+} |\psi_{\del^0K} - \psi_{e^0}|^2 \right)^{1/2}
%_{=:I^h_1}
\\
\, \left(C\sum_{K \in \Tau^h, e^0 \in \del^0K} \len \inf_u \del_u q_{e_K^+} | \tu - u_{e_K^+}|^2 \right)^{1/2} =: (I^h_1)^{1/2}  (I^h_2 )^{1/2}.
\end{multline}
Let $D:= \operatorname{supp} \psi$ then \eqref{tri1} and \eqref{tri2} imply
$$
I^h_1 \leq C c(D)^2 h^{-1} \sup_{K \in \Tau^h, e^0 \in \del^0K} |\psi_{\del^0K} - \psi_{e^0}|^2 .
$$
Now \eqref{tri1} and the fact that $\psi \in  \mathcal{D}(M)$ imply
$$
|\psi_{\del^0K} - \psi_{e^0}| \leq C h
$$ and hence
$$
I_1^h \rightarrow 0 \text{ for } h \rightarrow 0.
$$

In order to derive an estimate for $I^h_2$ we define
$$
L^h:= \{ K \in \Tau^h: K \cap D \not= \emptyset \} \text{ and } \del^0 L^h := \{ e^0 \in \del^0K : K \in L^h, e^0 \subset \del^0 M\},
$$
and note that  it is sufficient to find an estimate for
$$
\tilde I^h_2 :=  \sum_{K\in L^h, e^0 \in \del^0K} \len \inf_u \del_u q_{e_K^+} | \tu - u_{e_K^+}|^2
$$
instead of $I^h_2$. 
Using the first assertion of Lemma \ref{gede} we have
$$
\aligned
\tilde I_2^h &\leq C  \sum_{L^h} \left(q_{e_K^-}^\Omega(u_{e_K^-}) - q_{e_K^+}^\Omega(u_{e_K^+}) \right)
                                             - \sum_{e^0 \in \del^0 L^h} Q^\Omega_{K,e^0} (u_{e_K^-}, u_{e_{K_{e^0}}^-})  \\
& =: C( I^h_3 + I^h_4).
\endaligned
$$
In $I_3^h$ all terms which are inflow as well  as outflow faces cancel out and hence
$$
\aligned
I^h_3 = \sum_{K \in \Tau_0^h \cap L^h} q^\Omega_{e_K^-}(u_{e_K^-}) 
&\leq  \sum_{K \in \Tau_0^h \cap L^h} \int_{e_K^-} \int_0^{u_{e_K^-}} \del_v U(v) i^* \del_v \omega(v) \,dv
\\ 
  & \leq \|\del_u U \|_\infty  \sum_{K \in \Tau_0^h \cap L^h} \int_{e_K^-}|u_{e_K^-}| \alpha
\\ 
  &  \leq \|\del_u U \|_\infty \|u_0\|_\infty \sum_{K \in \Tau_0^h \cap L^h} \int_{e_K^-}  \alpha
 \leq C \|u_0\|_\infty\int_{\Hcal_0 \cap D} \alpha.
\endaligned
$$
Furthermore we get using \eqref{tri1} and \eqref{tri2}
$$
\aligned
I^h_4 &= - \sum_{e^0 \in \del^0 L^h} Q_{K,e^0}^\Omega (u_{e_K^-}, u_{e_{K_{e^0}}^-})\\
      &\stackrel{\text{Lemma } \ref{dbc}}\leq  \sum_{e^0 \in \del^0 L^h} \Big(\left| Q^\Omega_{K,e^0}(u_{e_{K_{e^0}}^-},u_{e_{K_{e^0}}^-}) \right| \\
  &\qquad \qquad+
\Big|\del_uU(u_{e_{K_{e^0}}^-}) \big[ Q_{K,e^0} (u_{e_K^-},u_{e_{K_{e^0}}^-}) - Q_{K,e^0}(u_{e_{K_{e^0}}^-},u_{e_{K_{e^0}}^-})\big]\Big| \Big)
\\
& \stackrel{\eqref{CFL}} \leq   \sum_{e^0 \in \del^0 L^h} \Bigg(
 \left| \int_{e^0} i^* \Omega(u_{e_{K_{e^0}}^-}) \right|
+ \|\del_u U\|_\infty \inf_u \del_u q_{e_K^+}(u)\left| u_{e_K^-} - u_{e_{K_{e^0}}^-}\right| \Bigg)
\\
& \leq \sum_{e^0 \in \del^0 L^h} \left(
\int_{e^0} \int_0^{u_{e_{K_{e^0}}^-}} \del_v U(v) i^*\del_v \omega(v) \, dv  \right)
  + C c(D)^2 h^{-1} h^{1-n} h^n \Big(\sup_{K \in \Tau^h} | u_{e_K^-}|  + \|u_B\|_\infty\Big)\\
&\leq C \|u_B\|_\infty \| \del_uU\|_\infty \int_{\del^0M \cap D} \alpha + C.
\endaligned
$$
Note that $\sup_{K \in \Tau^h} | u_{e_K^-}| \leq  \|u_B\|_\infty$ because of the CFL condition \eqref{CFL} and since $\omega$
is geometry compatible.
So finally $I^h_2$ is bounded and hence $A^h(\psi)$ goes to zero.

Let us now turn our attention to $B^h(\psi) $
\[
\begin{split}|B^h(\psi)|&= \Big|\sum_{K \in \Tau^h, e^0 \in \del^0K} \int_{e^0} (\psi_{e^0}-\psi) \phi^\Omega_{e^0} (u_{e_K^-}) \alpha_{e^0} \Big|\\
& = \Big|\sum_{K \in \Tau^h, e^0 \in \del^0K} \int_{e^0} (\psi_{e^0}-\psi) \big(\phi^\Omega_{e^0} (u_{e_K^-}) -\bar \phi^\Omega_{e^0} (u_{e_K^-})\big)  \alpha_{e^0} \Big|\\
& \leq \sum_{K \in L^h, e^0 \in \del^0K} \operatorname{diam}(e^0) o(1)  \int_{e^0} \alpha_{e^0} 
= o(1), 
\end{split}
\]
where we used the regularity of $\psi$, \eqref{curvcond}, and \eqref{tri1}$_2$ and \eqref{tri2} which imply  $\# \{ e^0 \in \del^0K \, : \, K \in L^h\} \leq C h^{-(n+1)}$.

The term $D^h(\psi)$ can be estimated as follows:
\[
 D^h(\psi) = - \sum_{K \in \Tau^h,e^0 \in \del^0K} \int_{e_K^+} \lambda_{K,e^0} ( \psi - \psi_{e_K^+}) \del_u U(u_{e_K^+}) ( i^* \omega (\tu) - i^*\omega(u_{e_K^+})), 
\]
where $ \psi_{e_K^+}$ is any average of $\psi$ on $e_K^+,$ because of \eqref{convdec}.
Using admissibility of $U$ and regularity of $\psi$ we obtain
\[ |D^h(\psi)| \leq C \sum_{K \in \Tau^h,e^0 \in \del^0K} \operatorname{diam}(e^+_K) \sup_u (\del_u q_{e_K^+}) |\tu - u_{e_K^+}|.\]
This bound for $| D^h(\psi)|$ can be estimated exactly as the bound for $|A^h|$ obtained in \eqref{531}.
The term $C^h(\psi)$ can be estimated analogously to $A^h(\psi)$ and $D^h(\psi)$.

Finally we consider $E^h(\psi)$ and note that by regrouping terms  we obtain
$$
\aligned
 E^h(\psi) 
&= \sum_{j=1}^\infty \sum_{e^0 \in \del \Tau_j^h} \Big( \int_{e_K^-} (\psi_{\del^0 K^-} - \psi) i^*\Omega (u_{e_K^-})
 -  \int_{e_K^+} (\psi_{\del^0 K} - \psi) i^*\Omega (u_{e_K^-})\Big)
\\
& \qquad - \sum_{e^0 \in \del \Tau_0^h} \int_{e_K^+}  (\psi_{\del^0 K} - \psi) i^*\Omega (u_{e_K^-})\\
&
 =: E_1^h(\psi) + E_2^h(\psi),
\endaligned
$$
where for $K \in \Tau_j^h$ we denote by $K^-$ the cell having $e_K^-$ as outflow face.
Due to the regularity of $\psi$ we have $E^h_2(\psi) \rightarrow 0$ for $h \rightarrow 0$.
Moreover, $E^h_1(\psi) \rightarrow 0$ for $h \rightarrow 0$ due to assumption \eqref{trichange}.

\end{proof}

For every triangulation $\Tau^h$ the finite volume method  \eqref{fvgm} generates an approximate solution of \eqref{LR.1},\eqref{LR.2i} defined by
\be
\label{apsol}
u^h(x) := u_{e_K^-} \text{ for } x \in K.
\ee

\begin{theorem}[Convergence of the finite volume schemes on a spacetime]\label{conv}
Let $\omega$ be a geometry compatible flux field with at most linear growth and satisfying the global hyperbolicity condition on a spacetime $M$ 
and  let $\{u^h\}_{h>0}$ be the sequence of approximate solutions generated by the finite volume method associated to a family of triangulations $\Tau^h$ satisfying \eqref{tri1}, \eqref{tri2}, \eqref{curvcond} and \eqref{trichange}.
Then $u^h$ converges to an entropy solution of the initial boundary value problem \eqref{LR.1},\eqref{LR.2i}
for $h \rightarrow 0$.

\end{theorem} 

The proof of Theorem \ref{conv} follows from Theorem \ref{MTHM.2} and  Lemma \ref{lemma:COP-1} below. 
The proof of Theorem~\ref{MTHM.2}  is omitted since it  follows along the same lines as in the Riemannian setting treated in \cite{BL07}, once charts are chosen which are compatible with the foliation.

As the CFL condition \eqref{CFL} implies $L^\infty$ stability of the finite volume scheme there 
exists a Young measure $\nu: M \rightarrow \text{Prob}(\RR),$ which allows us to determine all weak-$*$ limits of composite functions $a(u^h)$
for all continuous functions $a$, as $h \to 0$, 
\be
\label{COP.1}
a(u^h) \mathrel{\mathop{\rightharpoonup}\limits^{*}} \langle \nu, a \rangle
=
 \int_\RR a(\lambda) \, d\nu(\lambda). 
\ee

\begin{lemma}[Entropy inequalities for the Young measure]
\label{lemma:COP-1}
Let $\nu$ be a Young measure associated with the finite volume approximations $u^h$.
Then, for every convex entropy flux field $\Omega$ and every non-negative test-function $\psi \in \Dcal(M)$ 
supported compactly in $M \setminus \Hcal_T,$ 
 there exists a boundary field $\gamma \in \Lloc\Lambda^n(\del M)$
such that 
\be
\label{COP.2}
\aligned
& \int_M  \Big\la \nu,  d \psi \wedge \Omega(\cdot)  +  \psi \, \big((d  \Omega) (\cdot) - \del_u U(\cdot) (d \omega) (\cdot)\big) \Big\ra
\\
& - \int_{\del M} \psi_{|\del M} \, 
\Big( i^*\Omega(u_B) + \del_u U(u_B)  \big(\gamma - i^*\omega(u_B)\big)\Big) 
\,  \geq 0.
\endaligned
\ee
\end{lemma}

\begin{proof}
We start with the inequality from Lemma \ref{53}. Recalling
$(d(\psi \Omega))(u) = d\psi \wedge \Omega (u) + \psi (d\Omega)(u)$
 and $d\omega=0$ we see that the first summand of \eqref{53:eq} converges to 
\be 
- \int_M  \Big\la \nu,  d \psi \wedge \Omega(\cdot)  +  \psi \, \big((d  \Omega) (\cdot) - \del_u U(\cdot) (d \omega) (\cdot)\big)\Big\ra
\ee
for $h \rightarrow  0$.

We define $\gamma$ as the weak-$*$ limit of the sequence $\left\{ Q_{K,e^0}(u_{e_K^-},u_{e_{K_{e^0}}^-})\right\}$ on $\del^0 M$ and by
$\gamma= i^*\omega(u_B)$ on the rest of $\del M$. We know that $u_{e_{K_{e^0}}^-} \rightarrow u_B$ strongly and thus 
using the definition of $\gamma$ we see that the other summands on the left hand side of the inequality converge to
\be
\int_{\del M} \psi_{|\del M} \, 
\Big( i^*\Omega(u_B) + \del_u U(u_B)  \big(\gamma - i^*\omega(u_B)\big)\Big) 
.
\ee
Finally by Lemma \ref{54} the right hand side of the inequality converges to zero, which finishes the proof.
\end{proof}

\begin{proof}[Proof of Theorem \ref{main}]
The existence of entropy solutions follows from the convergence of the finite volume method, once we have shown that a family of triangulations on $M$ exists 
that satisfies \eqref{curvcond} and \eqref{trichange} for $h \rightarrow 0$.
First, we fix a Riemannian metric on $N$. 
This can be achieved by fixing an atlas $\{ U_i, \varphi_i\}_{i \in I}$ of $N,$ with a finite index set $I,$
and a partition of unity which is subordinate to the open cover $\{ U_i\}_{i \in I}$.
The pull-back of the Euclidean Riemannian metric under each chart $\varphi_i$ induces a Riemannian metric on $U_i$.
Using the partition of unity these Riemannian metrics can be combined to a Riemannian metric on $N$.
This immediately induces a Riemannian metric on $M=N\times [0,T]$. 

Let us now fix a triangulation on $N$ using so-called Karcher simplexes, \cite{vonDeylen}, denoted $\{ T_j\}_{j \in J}$ for a 
finite index set $J$.
For each of these simplexes the so-called barycentric map \cite[Def. 5.4; Thrm. 6.17]{vonDeylen} is a $C^2$-diffeomorphism from the standard simplex to $T_j$.
By subdividing the standard simplex using smaller simplexes we obtain subdivisions of the $T_j$ whose faces have bounded curvature (since the second derivatives of the finitely many barycentric maps are bounded).
If the subdivisions of all $T_j$ come from the same uniform subdivision of the standard simplex, then the ratios of  cell diameters and cell volumes will be uniformly bounded.
Thus, the triangulation $\Tau$ defined in \eqref{particular:tau}
will satisfy \eqref{tri1} and \eqref{tri2}.

Then, for some given (sufficiently small) $h>0,$  we define a triangulation $\Tau$ of $M$ by fixing a triangulation $\mathcal{S}$ of $N$ such that the maximal diameter 
of any element of $\mathcal{S}$ is bounded by $h/2$.
Then, we fix a $0 < \bar h < h$ such that  $T/\bar h$ is a natural number and the CFL condition is satisfied for the following triangulation:
\bel{particular:tau}
 \Tau :=\Big \{ \Big[\frac{(i-1)\bar h}{2}, \frac{i\bar h}{2}\Big] \times S \, : \, 1 \leq i \leq \frac{2T}{\bar h}, \ S \in \mathcal{S}\Big\}.
\ee
The temporal faces of this triangulation have bounded curvature such that \eqref{curvcond} is fulfilled, compare the first two items of Remark \ref{rem:triconds}.

In order to see that $\Tau$ also satisfies \eqref{trichange}, note that due to the construction of $\Tau$ for each cell $K$ the ``previous'' cell $K^-$ is just a translation of $K$ in time.
In particular, for any $K$ the faces $e_K^+$ and $e_K^-$ are (temporal) translations of each other.
Since all terms in the integrands in \eqref{trichange} depend smoothly on time we have
\[ \Big| \int_{e_K^-}(\psi_{\del^0 K^-} - \psi) i^*\Omega (u_{e_K^-})\\
 -  \int_{e_K^+} (\psi_{\del^0 K} - \psi) i^* \Omega (u_{e_K^-})\Big|  \leq \bar h \cdot h^{n+1} \leq h^{n+2}.\]
 Due to \eqref{tri2} we obtain
\[\sum_{K \in \Tau_j} \Big| \int_{e_K^-}(\psi_{\del^0 K^-} - \psi) i^*\Omega (u_{e_K^-})\\
 -  \int_{e_K^+} (\psi_{\del^0 K} - \psi) i^* \Omega (u_{e_K^-})\Big| \leq h^{2}.\]
 This completes the proof of existence of entropy solutions.
 
In order to prove the $L^1$ stability property \eqref{MTHM.1} we first need to determine $\gamma$ in Definition \ref{LR.4}.
We choose a test-function $\psi \in \mathcal{D}(\del M)$ with sufficiently smooth support and  a sequence $\{\phi_\varepsilon \}_{\varepsilon >0} \subset \mathcal{D}(M)$ such that
$\phi_\varepsilon \stackrel{\varepsilon\rightarrow 0}\longrightarrow \psi H_{\del M}$ in the sense of distributions,
where $H_{\del M}$ denotes the $n$-dimensional Hausdorff measure on $\del M$ induced by the metric chosen above.
Using $\phi_\varepsilon$ as test-function in the entropy inequality and letting $\varepsilon$ go to zero, we obtain
\bel{eq:bei}
 \int_{\del M} \psi\big( \Omega (u) - \Omega(u_B) - \del_u U(u_B) (\gamma - \omega (u_B))\big) \geq 0
\ee
for every admissible convex entropy pair $(U,\Omega)$.
By choosing $U(u)=\pm u $ and $\Omega (u)=\pm \omega (u)$ we obtain
\begin{equation}
 \int_{\del M} \psi\big( \pm \omega (u)  \mp  \gamma \big) \geq 0
\ee
which allows us to identify $\gamma=\omega(u)$.

Since $\psi$ was arbitrary we obtain
\bel{eq:biq}
 \int_{U} \big( \Omega (u) - \Omega(u_B) - \del_u U(u_B) (\omega(u)  - \omega (u_B))\big) \geq 0
\ee
for any open $U \subset \del M$.
Choosing Kruzkov's entropy pair equation \eqref{eq:biq} amounts to
\bel{eq:biqK}
 \int_{U}  \big( \sgn(u- \kappa) - \sgn(u_B - \kappa) \big) \big(\omega(u) - \omega(u_B) \big)  \geq 0.
\ee
By checking carefully several cases we obtain that \eqref{eq:biqK} and a similar inequality for $(v,v_B)$ 
imply the following inequality for traces on $\del M:$
\bel{eq:boundary}
 \Omegabf(u,v) \geq - | u_B - v_B | A_B
\ee
for any  $A_B \in L^\infty \Lambda^n (\del M)$ such that
$|\del_ u \omega |_{\del M}| \leq A_B$ in $ \del M \times \mathbb{R}$. 
The existence of one such $A_B$ follows since $\omega$ grows at most linearly.

By the classical doubling of variables method we obtain (in the interior of $M$)
\bel{eq:weakK}
 d \Omegabf(u,v) \leq 0 \quad \text{ weakly}.
\ee
Let us now pick hypersurfaces with $\Hcal, \Hcal'$  such that $\Hcal'$ lies in the future of $\Hcal$ with a corresponding function $F: \Hcal \times [0,1] \rightarrow M$.
Due to the non-degeneracy of $DF(\del_s)$ we can associate with each $x \in \operatorname{Im}(F)$ a ``time'' $s=s(x)$.
For each $\varepsilon>0$ we denote by 
$\chi_\varepsilon:M \rightarrow [0,\infty)$ the map given by
\[ \chi_\varepsilon (x) = \left\{ \begin{array}{ccc}
                              0, && x \not\in \operatorname{Im}(F),
\\
                             \frac{1}{\varepsilon} s(x), && x \in \operatorname{Im}(F),\ s(x) \leq \varepsilon,
\\
                             1, && x \in \operatorname{Im}(F),\ \varepsilon <  s(x) < 1-\varepsilon,
\\
                             \frac{1}{\varepsilon} (1 - s(x)), && x \in \operatorname{Im}(F),\ s(x)  \geq 1 -\varepsilon.
                             \end{array}\right.
\]
                             
Similarly we denote for any $\delta >0$ by $\psi_\delta \in \mathcal{D}(M)$ a  function satisfying
\[ \psi_\delta|_{\del M} \equiv 0, \quad \psi_\delta (x)=1 \quad \text{ for all } x \in M \, : \, \operatorname{dist}(x, \del M) > \delta.\]
This definition makes sense, once we have fixed a Riemannian metric on $M$.
Using $\phi=\chi_\varepsilon \cdot \psi_\delta$ as test-function for \eqref{eq:weakK} and letting $\delta \rightarrow 0$ we obtain
\begin{multline}\label{eq:pl1c}
 \frac{1}{\varepsilon} \int_{F(\Hcal \times [0,\varepsilon])} \Omegabf(u,v) \wedge F_*(d s) - \frac{1}{\varepsilon} \int_{F(\Hcal \times [1-\varepsilon,1])} \Omegabf(u,v) \wedge F_*(d s) 
 - \int_{\Bcal} \chi_\varepsilon \Omegabf(u,v) \geq 0,
\end{multline}
where $F_*$ denotes the push-forward along $F$.
Inserting \eqref{eq:boundary} into \eqref{eq:pl1c} implies
\begin{multline}\label{eq:pl1cab}
 \frac{1}{\varepsilon} \int_{F(\Hcal \times [1-\varepsilon,1])} \Omegabf(u,v) \wedge F_*(d s) 
 \leq \frac{1}{\varepsilon} \int_{F(\Hcal \times [0,\varepsilon])} \Omegabf(u,v) \wedge F_*(d s)
 + \int_{\Bcal } |u_B - v_B| A_B.
\end{multline}
This implies \eqref{MTHM.1} as soon as we can show that the $\varepsilon \rightarrow 0$ limit is well-defined.

To this end we define
\[ a_n :=  2^{n} \int_{F(\Hcal \times [0, 2^{-n}])} \Omegabf(u,v) \wedge F_*(d s), \quad b_n:= \int_{F(\Hcal \times [0, 2^{-n}]) \cap \del M } |u_B - v_B| A_B.\]
Then, choosing suitable test functions in \eqref{eq:weakK} we get
$
 a_{n+1} \geq a_n - b_n
$
and $|b_n| = \mathcal{O}(2^{-n}),$
which means that 
$\int_{\Hcal} \Omegabf(u,v) = \lim_{n \rightarrow \infty} a_n$ exists.
Similarly it can be shown that
\begin{equation}
  \frac{1}{\varepsilon} \int_{F(\Hcal \times [1-\varepsilon,1])} \Omegabf(u,v) \wedge F_*(d s)  \stackrel{\varepsilon \rightarrow 0}\longrightarrow \int_{\Hcal'} \Omegabf(u,v) ,
\ee
which completes the proof of \eqref{MTHM.1}.
The proof of \eqref{MTHM.23} is analogous.
 \end{proof}

%==============================================================================================================

\section*{Appendix}

We provide here examples showing that $\del^- M \not=\emptyset$ does not follow from the hyperbolicity condition 
and that
the existence of a foliation does not follow from the hyperbolicity condition and $\del^- M \not=\emptyset$.

\paragraph{Example 1} The assumption $(\del M)^- \not= \emptyset$ cannot be dropped, as it is essential for our finite volume scheme
and does not follow from the other properties of the problem under consideration, as can be seen in the following example. 

We consider the following compact and bounded manifold with boundary:
$$
R= \{ (x,y) \in \RR^2 : 1 \leq x^2 + y^2 \leq 2\}
$$
with the differentiable structure as a subset of $\RR^2$ and  the orientation determined by saying that $dx \wedge dy$ is positive.
Let us study the conservation law associated to
$$
\omega(x,y,u):= ux\,dx + uy\,dy.
$$
So, $\del_u \omega(x,y,u)=x\, dx + y\,dy,$ i.e., $\del_u \omega$ is  independent of $u$ and $\omega$  has at most linear growth.
Furthermore, $\omega$ is geometry compatible, since $(d\omega)(x,y,\bar u)=0$ holds for every $\bar u \in \RR$.
An observer can be defined by  $T(x,y):= y \, dx - x \, dy \in C^\infty\Lambda^1R,$ since
$$
T \wedge \del_u \omega= (x^2+y^2)\, dx \wedge dy 
$$
is a positive $2-$form. Hence, $\omega$ satisfies the global hyperbolicity property.
For every $(x,y) \in \del R$ we have
\[ T_{(x,y)} \del R = \{ \lambda(-y\, \del_x + x \, \del_y) | \lambda \in \RR\}.\]
Therefore normal $1$-forms have the form $N=\mu (x \,dx + y \, dy)$ with $\mu \in \RR$ and hence
$N \wedge \del_u \omega =0$, 
due to the anti-symmetry of the $\wedge$-operator on $1$-forms.
This shows that there is no spacelike, and in particular no inflow, part of the boundary.

\paragraph{Example 2} The existence of a  foliation consisting of homeomorphic, spacelike slices as in \eqref{foliation} is not a consequence of the other assumptions made on the problem under consideration.
We consider $M= \{ (x,y) \in \RR^2 \, | \, (x,y) \in [0,3]^2 \setminus (1,2)^2\}$ with the differentiable structure and orientation 
as a subset of $\RR^2$.
We fix a geometry compatible flux field and an observer that are given by  $\omega = - u \, dx$ and  $T= dy $.
The inflow boundary is $\del^- M = [0,3] \times \{0\} \cup [1,2] \times \{2\} \not =\emptyset$.
The spacetime $M$ cannot admit a foliation in the sense of \eqref{foliation}, i.e. one whose slices are spacelike with respect to $\omega$ and homeomorphic.
The reason is that whichever part of the inflow boundary $\del^- M$ we pick as $\Hcal_0$ the topological structure of $[0,T] \times \Hcal_0$ can never be that of $M$, i.e., connected but not simply connected.

%==========================================================================================================

\section*{Acknowledgments} 

The second author (PLF) was partially supported by the Innovative Training Network (ITN) grant 642768 (ModCompShock). Part of this work was done when the second author was visiting the Courant Institute for Mathematical Sciences, New York University.

%==========================================================================================================


\begin{thebibliography}{10} 


\bibitem{ABL} \auth{P. Amorim, M. Ben-Artzi, and P.G. LeFloch,}
Hyperbolic conservation laws on manifolds: Total variation estimates and finite volume method, 
Meth. Appl. Analysis 12 (2005), 291--324. 

\bibitem{ALO} \auth{P. Amorim, P.G. LeFloch, and B. Okutmustur}, 
Finite volume schemes on Lorentzian manifolds, 
Comm. Math. Sc. 6 (2008), 1059--1086. 

\bibitem{BLN} \auth{C.W. Bardos, A.-Y. Leroux, and J.-C. Nedelec},
First order quasilinear equations with boundary conditions, 
Comm. Part. Diff. Eqns. 4 (1979), 1017--1034.

\bibitem{BL} \auth{A. Beljadid and P.G. LeFloch,}
A central-upwind geometry-preserving method for hyperbolic conservation laws on the sphere,  Commun. Appl. Math. Comput. Sci. 12 (2017), no. 1, 81-–107.

\bibitem{BLM1} \auth{A. Beljadid, P.G. LeFloch, and M. Mohamadian,} 
A geometry-preserving finite volume method for conservation laws on the sphere, 
 Adv. Comput. Math. (ACOM). 

\bibitem{BFL09} \auth{M. Ben-Artzi, J. Falcovitz, and P.G. LeFloch},
Hyperbolic conservation laws on the sphere. A geometry-compatible finite volume scheme.
J. Comput. Phys. 228 (2009), 5650--5668.  

\bibitem{BL07} \auth{M. Ben-Artzi, and P.G. LeFloch},
   Well-posedness theory for geometry-compatible hyperbolic
              conservation laws on manifolds,
 Ann. Inst. H. Poincar\'e Anal. Non Lin\'eaire 24 (2007), 989--1008.

\bibitem{BCHL09} \auth{M.J. Berger, D.A. Calhoun, C. Helzel, and R.J. LeVeque},
 Logically rectangular finite volume methods with adaptive refinement on the sphere.
Philos. Trans. R. Soc. Lond. Ser. A Math. Phys. Eng. Sci. 367 (2009), 4483--4496.
 
\bibitem{CCL94} \auth{B. Cockburn, F. Coquel, and P.G. LeFloch},
Error estimates for finite volume methods for multidimensional conservation laws, 
Math. of Comput. 63 (1994), 77--103. 

\bibitem{CCL95} \auth{B. Cockburn, F. Coquel, and P.G. LeFloch},
Convergence of finite volume methods for multi-dimensional conservation laws,
SIAM J. Numer. Anal. 32 (1995), 687--705. 

\bibitem{CLF1} \auth{F. Coquel and P.G. LeFloch},
Convergence of finite difference schemes for scalar conservation laws in several space variables. 
The corrected antidiffusive-flux approach, 
Math. of Comput. 57 (1991), 169--210.

\bibitem{CLF2} \auth{F. Coquel, and P.G. LeFloch}, 
Convergence of finite difference schemes for scalar conservation laws in several space variables. General theory,
SIAM J. Numer. Anal. 30 (1993), 675--700.

\bibitem{DiPerna} \auth{R.J. DiPerna,}
Measure-valued solutions to conservation laws,
Arch. Rational Mech. Anal. 88 (1985), 223--270.

\bibitem{DL88}  \auth{F. Dubois and P.G. LeFloch},
Boundary conditions for nonlinear hyperbolic systems of conservation laws.
J. Differential Equations 71 (1988), 93--122.  

\bibitem{Gie09} \auth{J. Giesselmann},
A convergence result for finite volume schemes on Riemannian manifolds,
Math. Model. Numer. Anal. 43 (2009), 929--955.

\bibitem{GM14} \auth{J. Giesselmann and T. M\"uller,}
Geometric error of finite volume schemes for conservation laws on evolving surfaces, 
Numer. Math. 128 (2014), 489--516.

\bibitem{KL}  \auth{C. Kondo and P.G. LeFloch,}
Measure-valued solutions and well-posedness of multi-dimensional conservation laws in a bounded domain, Portugal. Math. 58 (2001), 171--194. 

\bibitem{KMS15} \auth{D. Kr\"oner, T. M\"uller, and L.M. Strehlau,}
Traces for functions of bounded variation on manifolds with applications to conservation laws on manifolds with boundary, 
SIAM J. Math. Anal. 47 (2015), 3944--3962.

\bibitem{Kruzkov} \auth{S. Kruzkov,}
First-order quasilinear equations with several space variables,
Math. USSR Sb. 10 (1970), 217--243.  
 
\bibitem{LM14} \auth{P.G. LeFloch and  H. Makhlof,}
A geometry-preserving finite volume method for compressible fluids on Schwarzschild spacetime,
Commun. Comput. Phys. 15 (2014), 827--852.
 
\bibitem{LeFlochOkutmustur2}  \auth{P.G. LeFloch and B. Okutmustur,} 
Hyperbolic conservation laws on spacetimes. A finite volume scheme based on differential forms,
Far East J. Math. Sci. 31 (2008), 49--83.

\bibitem{LX16}
\auth{P.G. LeFloch and S. Xiang,}
 Weakly regular fluid flows with bounded variation on the domain of outer communication of a Schwarzschild black hole spacetime,
J. Math. Pures Appl.  106 (2016), 1038 -- 1090.


\bibitem{LM13} \auth{D. Lengeler and T. M\"uller},
Scalar conservation laws on constant and time-dependent Riemannian manifolds,
J. Differential Equations 254 (2013), 1705--1727.

\bibitem{Otto} \auth{F. Otto}, 
Initial-boundary value problem for a scalar conservation law, 
C.R. Acad. Sci. Paris Ser. I Math. 322 (1996), 729--734. 

\bibitem{Szepessy} \auth{A. Szepessy}, 
Measure-valued solutions of scalar conservation laws with boundary conditions,
Arch. Rational Mech. Anal. 107 (1989), 181--193.

\bibitem{Tadmor} \auth{E. Tadmor,}
Numerical viscosity and the entropy condition for conservative finite difference schemes, 
Math. Comput. 43 (1984), 369--381. 

\bibitem{vonDeylen} \auth{S.W. Von Deylen,} 
Numerical Approximation in Riemannian Manifolds by Karcher Means,
Ph.D. thesis, Freie Universit\"at, Berlin, 2014.
 % http://www.diss.fu-berlin.de/diss/receive/FUDISS\_thesis\_000000098074.

\end{thebibliography}
\end{document}